\documentclass[12pt]{amsart}

\usepackage{environ, xcolor}
\NewEnviron{commentA}{}

\newcommand\Aoff{\RenewEnviron{commentA}{}}
\Aoff{} 

\usepackage{amsmath, amssymb}
\usepackage{array}
\usepackage[frame,cmtip,arrow,matrix,line,graph,curve]{xy}
\usepackage{graphpap, color, paralist, pstricks}
\usepackage[mathscr]{eucal}
\usepackage[pdftex]{graphicx}
\usepackage[pdftex,colorlinks,backref=page,citecolor=blue]{hyperref}
\usepackage{cleveref}
\usepackage{tikz, verbatim}

\newtheorem{theorem}{Theorem}[section]
\newtheorem{proposition}[theorem]{Proposition}
\newtheorem{corollary}[theorem]{Corollary}
\newtheorem{lemma}[theorem]{Lemma}

\newtheorem{conjecture}[theorem]{Conjecture}

\theoremstyle{definition}
\newtheorem{definition}[theorem]{Definition}

\newtheorem{remark}[theorem]{Remark}

\usepackage[margin=1in]{geometry} 
\usepackage{amsmath,amssymb,mathtools, mathrsfs,esint,tikz-cd,verbatim}
\newcommand{\R}{{\mathbb R}}

\newcommand{\Z}{{\mathbb Z}}
\newcommand{\Q}{{\mathbb Q}}

\newcommand{\C}{{\mathbb C}}

\newcommand{\on}[1]{\operatorname{#1}}

\newcommand{\supp}{\on{supp}}

\newcommand{\vol}{\operatorname{vol}}

\newcommand{\mld}{\operatorname{mld}}

\newcommand{\p}{\mathrm{p}}

\newcommand{\CW}[1]{{\textcolor{blue}{[CW: #1]}}}


\title{Exceptional Fano varieties with small minimal log discrepancy}
\author{Louis Esser}
\author{Jihao Liu}
\author{Chengxi Wang}

\address{Department of Mathematics, Princeton University, Fine Hall, Washington Road, Princeton, NJ 08544-1000, USA}
\email{esserl@math.princeton.edu}

\address{Department of Mathematics, Northwestern University, 2033 Sheridan Rd, Evanston, IL 60208, USA}
\email{jliu@northwestern.edu}

\address{UCLA Mathematics Department, Box 951555, Los Angeles, CA 90095-1555, USA}
\email{chwang@math.ucla.edu}

\begin{document}

\begin{abstract}
We construct exceptional Fano varieties with the smallest known minimal log
discrepancies in all dimensions. These varieties are
well-formed hypersurfaces in weighted projective space. Their minimal log
discrepancies decay doubly exponentially with dimension, and achieve the 
optimal value in dimension $2$.
\end{abstract}

\maketitle

\tableofcontents

\section{Introduction}

We work over the field of complex numbers $\mathbb{C}$.

\noindent\textbf{Explicit birational geometry in high dimensions}. Explicit birational geometry, often referred to as the ``geography of birational geometry," studies the extreme values of algebraic invariants in birational geometry. Traditionally, the focus of explicit birational geometry has been on varieties of dimension at most $3$ with at worst canonical singularities. Recent research, however, has expanded this focus to include higher-dimensional varieties, those with singularities worse than canonical, and new invariants such as Tian's $\alpha$-invariant (also known as the global log canonical threshold in \cite{Totaro23}), minimal log discrepancies, and the value $N$ in the boundedness of $N$-complements.

A primary focus of recent studies is to construct examples with extreme invariants, either very small (when there is a positive lower bound) or very large (when 
there is an upper bound). Although these examples are not yet known to be optimal in higher dimensions, they provide essential insights that guide the setting of 
optimization goals, which is important even in lower dimensions. For instance, the surface examples constructed in \cite{AL19a} have been crucial for optimizations
in further studies \cite{LS23, LL23}. Thanks to the first author, the third author, B. Totaro, and others, numerous examples of varieties with extreme invariants 
have been established in arbitrary dimensions. These extreme values often show doubly exponential growth or decay with respect to the dimension of the ambient 
variety \cite{Ess24, ET23, ETW, ETW22, Totaro23, TW21, Wan23}. This paper continues this series of studies by focusing on the minimal log discrepancy of exceptional Fano
varieties.

\noindent\textbf{Explicit birational geometry for exceptional Fano varieties}. Exceptional Fano varieties are Fano varieties whose $\mathbb{R}$-complements are klt; equivalently, these are Fano varieties with $\alpha$-invariants strictly greater than 1 \cite{Bir21}. Birkar famously showed that exceptional Fano varieties form a bounded family \cite{Bir19}, which is one crucial step in his proof of the boundedness of $N$-complements and the BAB conjecture. Additionally, since exceptional Fano varieties are K-stable \cite{Tia87, OS12}, studying the explicit geometry of exceptional Fano varieties is greatly beneficial for the construction of K-moduli spaces.

In the context of explicit birational geometry, the following four invariants are crucial for characterizing an exceptional Fano variety $X$ of dimension $n$: the anti-canonical volume $\vol(-K_X)$, Tian's alpha invariant $\alpha(X)$, complement index (i.e. the smallest positive integer $N$ so that $N(K_X+B)\sim 0$ for some lc pair $(X,B)$), and the global minimal log discrepancy $\mld(X)$. Research on the explicit geometry of the first three invariants has been thorough:
\begin{itemize}
\item   The well-known explicit upper bound for $\vol(-K_X)$ is $(n+1)^n$ (cf. \cite[Theorem 3]{Liu18}), and there are known examples with $\vol(-K_X)$ doubly exponentially small; these examples are conjectured to attain the minimum \cite[Theorem 3.1, Conjecture 3.2]{Totaro23}.
\item  $\alpha(X)$ has a trivial lower bound of 1, and there are known examples where $\alpha(X)$ increases doubly exponentially with dimension \cite[Theorem 8.1, Question 8.2]{Totaro23}.
\item The complement index has a trivial lower bound of 1, and there are known examples where the complement index increases doubly exponetially with dimension \cite[Theorem 8.1, Question 8.2]{Totaro23}.
\end{itemize}
However, although the upper bound of $\mld(X)$ is trivially $1$, its explicit lower bound remains unclear. Even for surfaces, the optimal lower bound was proven only very recently by the second author and V. V. Shokurov \cite[Theorem 1.9]{LS23}. This paper addresses this gap in all dimensions by establishing an explicit expectation for the lower bound of $\mld(X)$.

\begin{theorem}\label{thm: main}
Let $s_n$ be the $n$th Sylvester number. Then there exists a sequence of exceptional Fano varieties $\{X_n\}_{n\geq 2}$, such that each $X_n$ is of dimension $n$,
$$\mld(X_n)=\frac{4(s_n-1)}{s_n^3-9s_n}$$
if $n$ is even, and
$$\mld(X_n)=\frac{4(s_n-3)}{s_n^3-19s_n+14}$$
if $n$ is odd. In particular, $\mld(X_n)$ is asymptotic 
to $4/s_n^2$, as $n \rightarrow \infty$, so that
this value converges to zero doubly exponentially
with $n$. 
\end{theorem}

The explicit construction of $X_n$ can be found in \Cref{thm:hypersurface} 
when $n$ is even and in \Cref{thm:hypersurface-odd} when $n$ is odd.
In each case, $X_n$ is a well-formed, but not quasismooth, hypersurface in a 
weighted projective space.
We conjecture that the values obtained in Theorem \ref{thm: main} 
are exactly the smallest mld's among all exceptional Fano varieties of the same dimension:

\begin{conjecture}\label{conj: smallest mld}
For any even (resp. odd) integer $n\geq 2$, the smallest minimal log discrepancy of an exceptional Fano variety of dimension $n$ is
$$\frac{4(s_n-1)}{s_n^3-9s_n}\left(\text{resp. }\frac{4(s_n-3)}{s_n^3-19s_n+14}\right).$$
\end{conjecture}

Conjecture \ref{conj: smallest mld} holds in dimension $2$ 
due to \cite[Theorem 1.9(2)]{LS23}, where the optimal value is $3/35$.

\noindent\textit{Idea of the proof of \Cref{thm: main}}. Similar to the construction of other extreme examples, the examples we construct in \Cref{thm: main} are well-formed hypersurfaces in weighted projective spaces.  The invariants central to some of those previous examples, such as anti-canonical volume and complement index, are easy to read off using the degree and weights. In contrast, there are two significant difficulties in proving \Cref{thm: main}: the computation of the minimal log discrepancy and the proof that the varieties are exceptional.
\begin{itemize}
\item \textbf{The computation of the minimal log discrepancy:} The singularities of a quasismooth hypersurface in a weighted projective space may be determined in a straightforward way from the degree and weights.  Since these are toric, the minimal log discrepancy is computable combinatorially (see \Cref{prop:toricmld}).  However, the computation of the mld at non-quasismooth points is not as straightforward, and we develop a method to resolve this issue. More precisely, we establish a formula for computing the minimal log discrepancies of hypersurfaces in toric varieties that satisfy certain properties, i.e., those defined by equations that are Newton non-degenerate (\Cref{thm:mldcomp}). This is essentially due to the fact that divisors computing the mlds in this case are obtained via toroidal resolutions. It is worth mentioning that we do not have similar formulas for arbitrary hypersurfaces. In our examples, we need to apply this result to the non-quasismooth point and also show that the resulting discrepancy value is indeed a global minimum. This involves lengthy but elementary arguments and is done in \Cref{hypersurfmld}.
\item \textbf{Proving that the varieties are exceptional:} The difficulty here is that there is no straightforward theoretical formula to control the lower bound of the $\alpha$-invariant, so the explicit structure of the variety needs to be used. A natural idea here is to use similar arguments as the weighted cone construction in \cite[Section 4]{Totaro23}. At the end of the day, this idea works, and the proof is completed in Theorems \ref{thm: exceptionality}
and \ref{thm: exceptionality-odd}. However, our arguments are more complicated than those in \cite{Totaro23}. This is because the hypersurfaces in \cite{Totaro23} have better inductive structure than ours, so we need to have more accurate control over the multiplicities at every point that may attain the $\alpha$-invariant.
\end{itemize}

Finally, we briefly mention here how we found the exceptional Fano varieties with small mlds. One key idea is that the exceptional del Pezzo surface with the smallest mld is already known \cite[Table 8]{LS23}. Although this surface was not created via a weighted hypersurface construction in \cite{LS23}, it has a very small volume \cite[Table 8]{LS23} and its singularities are explicitly characterized. Therefore, we can perform a thorough search for hypersurfaces $X$ in weighted projective surfaces and determine that
$$X_{282} \coloneqq \{x_0^2+x_1^3+x_2^{19}x_3+x_1x_2x_3^5=0\}\subset \mathbb{P}^3(141,94,13,35),$$
is the desired surface. This is the starting point of the construction of examples. The higher-dimensional examples are constructed by carefully comparing examples with those in previous literature and choosing weights.

\noindent\textit{Sketch of the paper}. In \Cref{sec: preliminaries}, we recall some preliminary definitions relating to subvarieties of weighted
projective space and singularities in birational geometry. In \Cref{sec: ld formula}, we establish a formula for the computation of minimal log discrepancies of hypersurface singularities in weighted projective spaces. In \Cref{sec: well-formed}, we construct the examples and prove their basic properties. In \Cref{sec: compute mld}, we compute the minimal log discrepancies of the examples and show that they are asymptotically proportional to $4/s_n^2$. In \Cref{sec: compute alpha}, we show that the examples we constructed are exceptional Fano, hence concluding the proof of \Cref{thm: main}.

\noindent\textbf{Acknowledgement}. The authors would like to thank Burt Totaro for useful discussions and comments.

\section{Preliminaries}\label{sec: preliminaries}

This section will introduce the necessary background on
subvarieties in weighted projective space and the concepts
from birational geometry that we use in the paper.

\subsection{Weighted Projective Varieties}

Let $c_0,\ldots,c_n$ be positive integers.
The \textit{weighted
projective space} $Y=\mathbb{P}(c_0,\ldots,c_n)$ is defined to be the quotient
variety $(\mathbb{A}^{n+1} \setminus \{0\})/\mathbb{G}_{\on{m}}$ over $\C$, 
where the multiplicative group $\mathbb{G}_{\on{m}}$ acts by
$t \cdot (x_0,\ldots,x_n)=(t^{c_0}x_0,\ldots,t^{c_n}x_n)$
\cite[section 6]{Iano-Fletcher}.
We can view
the weighted projective space $Y$ as the quotient stack
$\mathcal{Y}=[(\mathbb{A}^{n+1} \setminus \{0\})/\mathbb{G}_{\on{m}}]$; it
is a smooth Deligne-Mumford stack. 
We say that $Y$ is \textit{well-formed} if
the stack $\mathcal{Y}$ has trivial stabilizer in codimension 1,
or equivalently if $\gcd(c_0,\ldots,\widehat{c_j},\ldots,c_n)=1$
for each $j$. In the well-formed case, the canonical class
of the variety $Y$ is $K_Y=\mathcal{O}_Y(-\sum c_j)$.

When $Y$ is well-formed, 
we denote by $\mathcal{O}_Y(1)$ the sheaf associated with the standard
graded module shifted by $1$ over the graded coordinate ring of 
the weighted projective space $Y$ \cite[Tag 01M3]{StacksProject}.
It is
a reflexive sheaf associated to a Weil divisor. The divisor class 
$\mathcal{O}_Y(m)$ is Cartier if and only if $m$ is a multiple
of every weight $c_j$. We can also view $\mathcal{O}(1)$
as a line bundle on the stack $\mathcal{Y}$.
The intersection number 
$\int_{\mathcal{Y}} c_1(\mathcal{O}(1))^n$
is $1/(c_0\cdots c_n)$.
The \textit{degree} for an integral closed substack $T$ of dimension $r$
in $\mathcal{Y}$ means $\int_T c_1(\mathcal{O}(1))^r$.

Note that if $X$ is a hypersurface in $\mathcal{Y}$ defined by a weighted-homogeneous polynomial
of degree $d$,
then its degree as a substack of $\mathcal{Y}$ is $X\cdot 
c_1(\mathcal{O}(1))^{n-1}=
d/(c_0\cdots c_n)$.  

We say that the subvariety $X$ of $Y$
is \textit{well-formed} if $Y$ is well-formed
and $X \cap \mathrm{Sing}(Y)$ has codimension at least $2$ in $X$.
The subvariety $X$ is \textit{quasismooth} if the
affine cone over $X$ is smooth, where the affine cone is
the preimage of $X$ in $\mathbb{A}^{n+1} \setminus \{0\}$.

\subsection{Singularities in Birational Geometry}
The minimal log discrepancy, or mld, of a variety (or pair)
is a numerical measure of
its singularities.  We'll primarily be concerned with varieties in this
paper, but will state the definition for a pair $(X,D)$. 

\begin{definition}
Let $X$ be a normal complex variety and $D$ an effective $\Q$-divisor
on $X$ with the property that $K_X + D$ is $\Q$-Cartier.  Given a 
proper birational morphism $\mu: X' \rightarrow X$ from another 
normal variety $X$ and $E \subset X'$ any irreducible divisor,
the \textit{log discrepancy} of $E$ with respect to the morphism
$\mu$ is defined as:
$$a_E(X,D) \coloneqq \mathrm{ord}_E(K_{X'} + E - \mu^*(K_X+D)).$$
This definition turns out to only depend on the valuation defined
by $E$ on the function field of $X$, and not on the specific
choice of $\mu$ extracting this divisor.  The \textit{center} $c_X(E)$
of $E$ is $\mu(E) \subset X$. 
For any point $x$ of the scheme $X$, we define
$$\mld(X,D,x) \coloneqq \mathrm{inf}
\{a_E(X,D):c_X(E) = \bar{x}\},$$
where the infimum runs over all divisors $E$ on proper birational
models of $X$ with the indicated property.  Finally, the 
(global) minimal log discrepancy of the variety $X$ is
$$\mld(X,D) \coloneqq \mathrm{inf}_{x \in X} \mld(X,D,x).$$
where the infimum runs through all codimension $\geq 1$ points $x$ on $X$.
\end{definition}

Various classes of singularities in birational geometry can be 
defined using the minimal log discrepancy.

\begin{definition}
The pair $(X,D)$ as above is \textit{log canonical (lc)}  if 
$\mld(X,D) \geq 0$.  It is \textit{Kawamata log terminal (klt)}
if $\mld(X,D) > 0$.
\end{definition}

When a pair $(X,D)$ is log canonical, we may compute the mld
using only the finitely many irreducible divisors appearing in
a single log resolution of singularities of $(X,D)$; in particular,
it is rational (see, e.g., \cite{Kollar2013}).

\begin{definition}
A klt Fano variety $X$ is \textit{exceptional} if for every effective
$\mathbb{Q}$-divisor $D$ that is $\mathbb{Q}$-linearly equivalent to $-K_X$,
the pair $(X,D)$ is klt. 

Let $X$ be a klt Fano variety. The \textit{lc threshold} of an effective
$\mathbb{Q}$-Cartier $\mathbb{Q}$-divisor $D$ on $X$ is defined to be
\[\mathrm{lct}(X,D):=\mathrm{sup}\{t\in \mathbb{Q}:(X,tD) \text{ is lc}\}.\] 
Let $|-K_X|_{\mathbb{Q}} \coloneqq \{D \text{ is effective}| D\sim_{\mathbb{Q}}-K_X\}$
be the $\mathbb{Q}$-linear 
system of $-K_X$. The \textit{$\alpha$-invariant} (also called 
the global log canonical threshold) of $X$ is defined to be 
\[\alpha(X):=\mathrm{inf}\{\mathrm{lct}(X,D):D\in|-K_X|_{\mathbb{Q}}\}\] 
which coincides with \[\mathrm{sup}\{t\in \mathbb{Q}:(X,tD)
\text{ is lc for all } D\in |-K_X|_{\mathbb{Q}}\}.\]
\end{definition}

So the $\alpha$-invariant in this paper is same as the global log canonical 
threshold in \cite{Totaro23}, but different from the global log canonical
threshold in \cite{LS23}. From the definitions, we see that 
$\alpha$-invariant $\alpha(X)>1$ implies $X$ is exceptional.
The converse holds by \cite[Theorem 1.7]{Bir21}.

\section{Log Discrepancies of Hypersurface Singularities}\label{sec: ld formula}

In this section, we develop the tools needed to compute
minimal log discrepancies later in the paper.  The main result
is a method of computing the minimal log discrepancy of 
the quotient of a hypersurface in affine space by a finite
subgroup of the torus; this method requires that
the defining equation for the hypersurface satisfy a 
Newton non-degeneracy condition.

First, we recall a simple formula for the minimal log discrepancy of
a cyclic quotient singularity, which
follows from the toric description of the
singularity (see, e.g., \cite[Section 2]{Ambro}):

\begin{proposition}
\label{prop:toricmld}
Let $X:=\frac{1}{r}(b_1,\ldots,b_s)$ denote the quotient
$\mathbb{A}^s/\mu_r$, where $\mu_r$ acts on $\mathbb{A}^s$ with
the indicated weights.  Suppose that the singularity is well-formed
in the sense that $\gcd(r,b_1,\ldots,\widehat{b_i},\ldots,b_s)$ $= 1$ for each $1 \leq i \leq s$.
Then the minimal log discrepancy of $\frac{1}{r}(b_1,\ldots,b_s)$ is given by
\begin{equation}
\label{toricmldeq}
    \mld(X) = \min\left\{1 ,\min_{1 \leq j < r} \sum_{i = 1}^s 
    \left( \frac{jb_i}{r} - \bigg\lfloor \frac{jb_i}{r}\bigg\rfloor\right)\right\}.
\end{equation}
\end{proposition}

Note that this is the global mld of this quotient indicated, not the mld
of the point $0 \in \mathbb{A}^s/\mu_r$.

\begin{proof}
Since the quotient singularity is well-formed, no divisor in $\mathbb{A}^s$
is fixed by a nontrivial subgroup of $\mu_r$; hence the pullback of the 
canonical divisor by the quotient morphism is again the canonical divisor.
The quotient $\mathbb{A}^s/\mu_r$
is the affine toric variety corresponding to the cone $\sigma$ 
generated by the standard basis vectors in $\R^s$, with 
toric lattice $N = \Z^s+ \Z \cdot \frac{1}{r}(b_1,\ldots,b_s)$. 
The global minimal log discrepancy of $\mathbb{A}^s/\mu_r$ is therefore
the minimum of the sum-of-coordinates functional over all points in 
$N \cap \sigma \setminus \{0\}$ \cite[Section 2-A(f)]{Ambro2}.  
Actually, it suffices to take the minimum over lattice points 
in the unit cube $[0,1]^s$ (excluding the origin).

As $j$ runs from $1$ to $r-1$, the point with coordinates 
$$\frac{jb_i}{r} - \bigg\lfloor \frac{jb_i}{r} \bigg\rfloor$$
varies over 
every non-integral point in this unit cube.  
The minimum is taken with $1$ to account for discrepancies of toric
divisors on the original variety $\mathbb{A}^s/\mu_r$, 
which correspond to the standard basis points.
\end{proof}

Next, we develop a method for computing the minimal log discrepancies
of certain hypersurfaces in toric varieties. 
The idea is to refine the fan of the ambient
toric variety in such a way that the strict transform of the hypersurface
is smooth. The method of resolving hypersurfaces in this way
is well-studied and dates back at least to work of Varchenko
\cite{Varchenko} in the case of hypersurfaces of
affine space.  This procedure
works in much greater generality, however, and has been extended to
subvarieties defined by suitable ideals in any affine normal toric 
variety \cite{AGM}.  We will be concerned with hypersurfaces in 
the toric variety $\mathbb{A}^{n+1}/G$, where $G$ is a finite
cyclic subgroup of the torus acting
freely in codimension $1$.

Suppose that $Y \coloneqq \{f = 0\} \subset \mathbb{A}^{n+1}$ is a hypersurface
passing through the origin and $G \coloneqq \mu_r$ is a finite cyclic
group which acts by 
$\zeta \cdot (x_0,\ldots,x_n) = 
(\zeta^{a_0} x_0,\ldots,\zeta^{a_n} x_n)$,
for any $\zeta\in \mu_r$.  Assume further that the action of $G$
is free in codimension $1$ and preserves $Y$.

The toric variety $\mathbb{A}^{n+1}/G$ corresponds to the fan $\Delta$
in the lattice $N = \Z^{n+1} + \Z \cdot \frac{1}{r}(a_0,\ldots,a_n)$
whose cones are the positive orthant $\sigma$ in 
$N_{\R} \coloneqq N \otimes_{\Z} \R$ and all its faces.  Let $M$ be the 
dual lattice to $N$.  Then we may write $f = \sum_{m \in M}c_m x^m$ (each
$m \in M$ corresponds to a $G$-invariant monomial in $x_0,\ldots,x_n$).
Denote by $\supp(f)$ the finite set of $m \in M$ for which $c_m \neq 0$.

Our methods will apply to hypersurfaces defined by suitably general
polynomials $f$; the correct notion is Newton non-degeneracy
(cf. \cite[Definition 4.4.22]{Ishii}, \cite[Definition 1.2]{AGM}):

\begin{definition}\label{define-Newtonpolytope}
The \textit{Newton polytope} $\Gamma_+(f)$ of $f$ in 
$M_{\R} \coloneqq M \otimes_{\Z} \R$ is defined to be 
$$\Gamma_+(f) \coloneqq \on{conv}\left\{\bigcup_{m \in \supp(f)} (m +
\R_{\geq 0}^{n+1}) \right\}.$$
We say that $f$ is \textit{Newton non-degenerate at $0$}
if for every compact face 
$\gamma$ of $\Gamma_+(f)$, the polynomial $f_{\gamma} \coloneqq \sum_{m \in \gamma} 
c_m x^m$ satisfies $\mathrm{Sing}(V(f_{\gamma})) \cap T = \emptyset$,
where $T$ is the open torus defined by $x_0 \cdots x_n \neq 0$.
(We note that the definitions of Newton polytope and Newton non-degeneracy
do not depend on $G$, but only on the polynomial $f$.)
\end{definition}

Given a point $\beta \in N$ and a polynomial $f \in \C[M \cap \sigma^{\vee}]$, we
write $\beta(f) \coloneqq \min\{\beta(m): m \in \supp(f)\}$, where
$\beta(m)$ denotes the inner product. 
Now we can state the main result of this subsection:

\begin{theorem}
\label{thm:mldcomp}
Let $(Y,0) \coloneqq \{f = 0\} \subset \mathbb{A}^{n+1}$ be a normal hypersurface
singularity and $G \coloneqq \mu_r$ a cyclic subgroup
of the torus in $\mathbb{A}^{n+1}$ acting freely 
in codimension $1$.  Assume that $f$ is Newton 
non-degenerate at $0$, set $(X,x) \coloneqq (Y,0)/G$, and suppose
$(X,x)$ is well-formed in the sense that
$\mathrm{Sing}(\mathbb{A}^{n+1}/G) \cap X$ has codimension at least $2$ in $X$.  Then 
$$\mld(X) = \inf_{\beta \in (N \cap |\Delta|) \setminus \{0\}} [\beta(x_0 \cdots x_n) - \beta(f)].$$
\end{theorem}

Note that this result computes the global mld of $X$, not just the mld at
$x$.  Moreover, the singularity need not be isolated, though we require
that the origin $0$ be contained in the hypersurface
 $\{f = 0\} \subset \mathbb{A}^{n+1}$.
Before proving this result, we state a quick corollary, which gives
an easy way to check that $X$ is klt 
(cf. \cite[Proposition 2.9]{IshiiProkhorov}).

\begin{corollary}
\label{cor:klt_hypersurf}
Let $(X,x)$ be a singularity satisfying the conditions of \Cref{thm:mldcomp}.
If $(1,\ldots,1)$ is contained in the interior of the Newton polytope
$\Gamma_+(f)$, then $X$ is klt.
\end{corollary}

\begin{proof}
Let $\beta \in (N \cap |\Delta|) \setminus \{0\}$ be a lattice point.
Notice that $\beta(x_0 \cdots x_n)$ is the same as the value of the
functional $\beta: M_{\R} \rightarrow \R$ at the point $(1,\ldots,1)$.
By assumption, $\beta$ has all nonnegative coordinates in $N_{\R}$, so
the minimum value of $\beta$ on the Newton polytope $\Gamma_+(f)$
must be attained by some monomial $m$ of the equation $f$.
Because $(1,\ldots,1)$ is in the interior, we must have that
$\beta(x_0 \cdots x_n) - \beta(m)$ is positive.  Therefore,
$\beta(x_0 \cdots x_n) - \beta(f)$ is positive.  Since this holds
for each $\beta$, it follows from \Cref{thm:mldcomp} that $X$
is klt.
\end{proof}

\begin{proof}[Proof of \Cref{thm:mldcomp}]
Since $f$ is Newton non-degenerate, it follows from \cite[Theorem 1.3]{AGM}
that there exists a refinement $\Delta'$ of the fan $\Delta$, corresponding to a 
toric birational morphism 
$\pi_{\Delta'}: V' \rightarrow V \coloneqq \mathbb{A}^{n+1}/G$, 
such that $V'$ is smooth and the irreducible 
components of the total transform $\pi_{\Delta'}^{-1}(X)$ are smooth and meet
transversely.  In particular, the strict transform $\tilde{X}$ of $X$ is 
smooth. More specifically, 
the fan $\Delta'$ can be any refinement of the dual fan $\Sigma$ to
$\Gamma_+(f)$ for which the corresponding toric variety is smooth.
This dual fan $\Sigma$ is called the \textit{Groebner fan} in \cite{AGM} (cf. 
\cite[Theorem 4.4.23]{Ishii}).  In particular, since any two
refinements have a further common refinement, we can arrange that
$\Delta'$ also refines any other subdivision of $\Delta$.

Since $\tilde{X}$ is a resolution of singularities of $X$, when
$X$ is lc, we can compute its mld using only exceptional divisors
on $\tilde{X}$. In particular, in this case we only need to 
check the log discrepancies of exceptional divisors that arise from toric
modifications of $\mathbb{A}^{n+1}/G$.  Fix a primitive element $\beta$
in $(N \cap |\Delta|) \setminus \{0\}$,
the ray through which we may assume is in $\Delta'$.
We may refine $\Delta$ with a star-shaped subdivision $\Delta(\beta)$
centered at the ray through $\beta$; this gives a birational morphism 
$\pi_{\beta}: V_{\beta} \rightarrow V$ extracting a
unique exceptional divisor $E_{\beta}$.
Choosing $\Delta'$ to also refine $\Delta(\beta)$, we have in our setup
successive refinements
$\Delta \subset \Delta(\beta) \subset \Delta'$ and corresponding 
birational morphisms $V' \rightarrow V_{\beta} \rightarrow V$.

Denote by $X_{\beta}$ the strict transform of $X$ in $V_{\beta}$.
Then \cite[Proposition 8.3.11]{Ishii} shows that 
$$K_{V_{\beta}} = \pi_{\beta}^*(K_V) + (\beta(x_0 \cdots x_n) - 1)E_{\beta}$$
and
$$X_{\beta} = \pi_{\beta}^* X - {\beta}(f)E_{\beta}.$$
Combining these two equations gives that 
$$K_{V_{\beta}} + X_{\beta} + E_{\beta} 
= \pi_{\beta}^*(K_V + X) + (\beta(x_0 \cdots x_n) - \beta(f))E_{\beta}.$$
Passing to the full refinement $\Delta'$, a very similar equation still
holds in a neighborhood of the generic point of
the strict transform $\tilde{E}_{\beta}$ of $E_{\beta}$:
$$K_{V'} + \tilde{X} + \tilde{E}_{\beta} = \pi^*_{\Delta'}(K_V + X) + 
(\beta(x_0 \cdots x_n) - \beta(f))\tilde{E}_{\beta}.$$
Since $X$ is well-formed and normal, the adjunction formulas $(K_V + X)|_X = 
K_X$ and $(K_{V'} + \tilde{X})|_{\tilde{X}} = K_{\tilde{X}}$ hold. 
Indeed, the canonical class of a normal variety 
is well-defined as a Weil divisor up to linear equivalence,
so we can remove the singular locus from $X$,
which has codimension at least $2$ in $X$ by the 
well-formedness assumptions,
and observe that the equalities of divisors hold on the smooth locus.
Thus, after restricting to $\tilde{X}$, we obtain that the log discrepancy of
each component of $\tilde{E}_{\beta} \cap \tilde{X}$
(which is an exceptional divisor of $\tilde{X} \rightarrow X$, as 
$\tilde{E}_{\beta}$
intersects transversely with $\tilde{X}$)
is $\beta(x_0 \cdots x_n) - \beta(f)$.  

Suppose $\beta(x_0 \cdots x_n) - \beta(f)$ is negative for some 
$\beta$.  Since this is the log discrepancy of some divisor
over $X$, it follows that $X$ is non-lc, so that the 
minimal log discrepancy is $-\infty$ 
(see, e.g., \cite[Corollary 2.31]{KM}).  This agrees with the 
infimum in \Cref{thm:mldcomp}, for we may replace $\beta$ with
positive multiples to make the difference
$\beta(x_0 \cdots x_n) - \beta(f)$ arbitrarily negative.
Otherwise, the infimum is a minimum and is 
nonnegative, so in particular all the log discrepancies of 
divisors on any resolution $\tilde{X}$ constructed as above
are nonnegative.  This means $X$ is lc.
Since we can compute the mld in this case on just one resolution
$\tilde{X}$, the minimum of $\beta(x_0 \cdots x_n) - \beta(f)$
must be the the minimal log discrepancy.
\end{proof}

\section{The examples and their well-formedness}\label{sec: well-formed}

The exceptional Fano varieties used to prove \Cref{thm: main} are 
hypersurfaces in certain weighted projective spaces with weights
defined using \textit{Sylvester's sequence}.  This sequence is defined by
$s_0 \coloneqq 2$ and $s_{n+1} \coloneqq s_0 \cdots s_n + 1, n \geq 1$.
The first several
terms are $2,3,7,43,1807,\ldots$.
All elements of the sequence are pairwise coprime
and satisfy $s_n > 2^{2^{n-1}}$.  Their reciprocals
also have the following property for each $n$:
$$\frac{1}{s_0} + \cdots + \frac{1}{s_n} = 1 - \frac{1}{s_{n+1}-1}.$$
The hypersurfaces have slightly different
equations in even vs odd dimensions. 
\Cref{thm:hypersurface} and \Cref{thm:hypersurface-odd}
summarize their properties in each case.

\begin{theorem}\label{thm:hypersurface}
For every even integer $n \geq 2$, let 
\begin{align*}
    & a_n  \coloneqq \frac{1}{4}(s_n^2 + s_n - 4), \\
    & a_{n+1} \coloneqq \frac{1}{2}((s_n-1)a_n - s_n -1), \\
    & d \coloneqq (-1 + a_n + a_{n+1})(s_n-1), \\
    & a_i \coloneqq \frac{d}{s_i}, i = 0,\ldots,n-1, \\
    & b \coloneqq \frac{1}{2}(s_n^2 - s_n - 4), \\
    & c \coloneqq \frac{1}{2}(s_n+3),
\end{align*} and let $X$ be the hypersurface in $\mathbb{P}=\mathbb{P}(a_0,\ldots,a_{n+1})$ defined by the degree $d$ equation
\begin{equation}
\label{eveneq}
   x_0^2 + x_1^3 + \cdots + x_{n-1}^{s_{n-1}} + x_n^b x_{n+1} + x_1 \cdots x_{n-1} x_n x_{n+1}^c = 0. 
\end{equation}
Then $X$ is a well-formed exceptional Fano variety with $-K_X=\mathcal{O}_X(1)$ and
$$\mld(X) = \frac{s_n-1}{2a_{n+1}}.$$
\end{theorem}

\begin{theorem}
\label{thm:hypersurface-odd}
For every odd integer $n \geq 3$, Let 
\begin{align*}
    & a_n  \coloneqq \frac{1}{4}(s_n^2 + 3s_n - 6), \\
    & a_{n+1} \coloneqq \frac{1}{4}((s_n-3)a_n - s_n -1), \\
    & d \coloneqq (-1 + a_n + a_{n+1})(s_n-1), \\
    & a_i \coloneqq \frac{d}{s_i}, i = 0,\ldots,n-1, \\
    & b \coloneqq \frac{1}{4}(s_n^2 - s_n - 2), \\
    & c \coloneqq \frac{1}{2}(s_n+5).
\end{align*} and let $X$ be the hypersurface in $\mathbb{P}=\mathbb{P}(a_0,\ldots,a_{n+1})$ defined by a degree $d$ equation
\begin{equation}
\label{oddeq}
    x_0^2 + x_1^3 + \cdots + x_{n-1}^{s_{n-1}} + x_n^{b}x_{n+1} + x_1 \cdots x_{n-1} x_n^2 x_{n+1}^c = 0.
\end{equation}
Then $X$ is a well-formed exceptional Fano variety with $-K_X=\mathcal{O}_X(1)$ and
$$\mld(X) = \frac{s_n-3}{4a_{n+1}}.$$
\end{theorem}

We note that substituting in the definition of $a_{n+1}$ in the 
expressions for the mld yields the form given in \Cref{thm: main}.

\begin{remark}
\label{rem:coefficients}
In the two theorems above, we choose equations with coefficients
all equal to $1$ to define $X$.  However, since each equation
has the same number of monomials as variables, we could in fact
take the coefficients to be general nonzero complex numbers.
This is because any hypersurface of this form is isomorphic
to $X$ by scaling the
variables.
\end{remark}

The hypersurfaces in Theorem \ref{thm:hypersurface} and Theorem 
\ref{thm:hypersurface-odd} fail to be quasismooth at
the coordinate point of the last variable $x_{n+1}$.
We conjecture that each has the
smallest mld among all exceptional Fano varieties
of the same dimension $n$ (Conjecture \ref{conj: smallest mld}).
As a comparison, \cite[Theorem 3.1]{Totaro23} defines
non-quasismooth hypersurfaces with similar equations in weighted projective
spaces with different weights; these are conjectured to have
minimum anti-canonical volume among all exceptional Fano varieties of dimension $n$.

When $n = 2$, this example is the hypersurface $X_{282} \subset \mathbb{P}(141,94,13,35)$ defined by the equation
$$x_0^2 + x_1^3 + x_2^{19} x_3 + x_1 x_2 x_3^5 = 0.$$
The mld of this surface is $3/35$.
This is known to be the smallest possible mld of an exceptional
del Pezzo surface by \cite[Theorem 1.9(2)]{LS23}.  In particular, this
construction realizes the example in \cite{LS23} as a non-quasismooth
hypersurface in weighted projective space.

When $n = 3$, this example is the hypersurface 
$X_{227262} \subset \mathbb{P}^4(113631,75754,32466,493,4919)$
defined by the equation
$$x_0^2 + x_1^3 + x_2^7 + x_3^{451} x_4 + x_1 x_2 x_3^2 x_4^{24} = 0.$$
The mld of this threefold is $10/4919$.
The best known explicit lower bound on the mld of an
exceptional Fano threefold is $1/I_0$, where
$I_0 \approx 10^{10^{10.8}}$ \cite[Theorem 1.4]{BL23}.
We expect the example above to achieve the optimal bound.
To the authors' knowledge, the Fano threefold with
smallest mld previously known was Totaro's exceptional
threefold of small anti-canonical volume \cite[p. 8]{Totaro23}.
That example has mld $41/17629 > 10/4919$.

When $n = 4$, this example is a hypersurface in a certain
weighted projective space of dimension $5$
defined by the equation
$$x_0^2 + x_1^3 + x_2^7 + x_3^{43} + x_4^{1631719}x_5 + x_1 x_2 x_3 x_4 x_5^{905} = 0.$$
The mld of this fourfold is $903/737536085$.

\begin{proof}[\bf Proof of Theorem \ref{thm:hypersurface} and Theorem 
\ref{thm:hypersurface-odd}]
The proof of the properties of the examples defined in the 
two theorems will occupy the next 
few sections.
By \Cref{lem:first_properties}, the polynomials defining the hypersurfaces 
have degree $d$. By \Cref{well-formed}, the hypersurfaces are well-formed. 
\Cref{hypersurfmld} computes the mld. \Cref{thm: exceptionality} and 
\Cref{thm: exceptionality-odd} show the hypersurfaces are exceptional.
\end{proof}

We'll begin by confirming that the weights and equations above
do in fact define a hypersurface of the indicated degree.

\begin{lemma}
\label{lem:first_properties}
For each $n \geq 2$, $a_0 + \cdots + a_{n+1} = d + 1$ and the polynomial in \eqref{eveneq} or \eqref{oddeq} is homogeneous of degree $d$.
\end{lemma}
\begin{proof}

First, the sum of the weights is $d + 1$ for all $n$:
\begin{align*}
    a_0 + \cdots + a_{n+1} & = \frac{d}{s_0} + \cdots + \frac{d}{s_{n-1}} + a_n + a_{n+1} \\ 
    & = \frac{s_n-2}{s_n-1}(-1 + a_n + a_{n+1})(s_n-1) + (-1 + a_n + a_{n+1}) +1 \\ 
    & = (s_n-2+1)(-1 + a_n + a_{n-1}) + 1 = d + 1.
\end{align*}
It's clear that $x_0^2, x_1^3, \ldots, x_{n-1}^{s_{n-1}}$ are monomials of 
degree $d$ in both the even and odd cases.  We need to check that the last 
two monomials have degree $d$.  We begin with the case of $n$ even. 
For the monomial $x_n^b x_{n+1}$, observe that
\begin{align*}
    d - a_{n+1} &  = (s_n-1)a_n + (s_n-2)a_{n+1} - s_n+1 \\
    & = (s_n-1)a_n + (s_n-2)\frac{1}{2}((s_n-1)a_n - s_n - 1) - s_n + 1 \\
    & = (s_n-1)a_n + \frac{1}{2}(s_n^2 - 3s_n + 2)a_n - \frac{1}{2}(s_n^2 - s_n - 2) - s_n + 1 \\
    & = \frac{1}{2}(s_n^2 - s_n)a_n - \frac{1}{2}(s_n^2 + s_n - 4) \\
    & = \frac{1}{2}(s_n^2-s_n)a_n - 2a_n = \frac{1}{2}(s_n^2 - s_n - 4)a_n = ba_n.
\end{align*}
Finally, for the monomial $x_1 \cdots x_{n-1} x_n x_{n+1}^c$, note that:
\begin{align*}
    d - a_1 - \cdots - a_n & = \left(1 - \frac{1}{s_1} - \cdots - \frac{1}{s_{n-1}} \right)d - a_n \\
    & = \left(\frac{1}{2} + \frac{1}{s_n-1}\right)d - a_n \\
    & = \frac{s_{n}+1}{2}a_{n+1} + \frac{s_n - 1}{2}a_n - \frac{s_n+1}{2} \\
    & = \frac{s_n+1}{2}a_{n+1} + \frac{1}{2}((s_n-1)a_n - s_n - 1) \\
    & = \frac{s_n+1}{2}a_{n+1} + a_{n+1} \\
    & = \frac{1}{2}(s_n+3)a_{n+1} = c a_{n+1}.
\end{align*}
This completes the proof that the equation in \eqref{eveneq} is weighted homogeneous of degree $d$ for $n$ even. Next, we treat the case of odd $n$.
For the monomial $x_n^b x_{n+1}$, we have 
\begin{align*}
    d - a_{n+1} &  = (s_n-1)a_n + (s_n-2)a_{n+1} - s_n+1 \\
    & = (s_n-1)a_n + (s_n-2)\frac{1}{4}((s_n-3)a_n - s_n - 1) - s_n + 1 \\
    & = (s_n-1)a_n + \frac{1}{4}(s_n^2 - 5s_n + 6)a_n - \frac{1}{4}(s_n^2 - s_n - 2) - s_n + 1 \\
    & = \frac{1}{4}(s_n^2 - s_n + 2)a_n - \frac{1}{4}(s_n^2 + 3s_n - 6) \\
    & = \frac{1}{4}(s_n^2-s_n + 2)a_n - a_n = \frac{1}{2}(s_n^2 - s_n - 2)a_n = ba_n.
\end{align*}
Finally, for the monomial $x_1 \cdots x_{n-1} x_n^2 x_{n+1}^c$, note that
\begin{align*}
    d - a_1 - \cdots - a_{n-1} - 2a_n & = \left(1 - \frac{1}{s_1} - \cdots - \frac{1}{s_{n-1}} \right)d - 2a_n \\
    & = \left(\frac{1}{2} + \frac{1}{s_n-1}\right)d - 2a_n \\
    & = \frac{s_n+1}{2}a_{n+1} + \frac{s_n - 3}{2}a_n - \frac{s_n+1}{2} \\
    & = \frac{s_n+1}{2}a_{n+1} + \frac{1}{2}((s_n-3)a_n - s_n - 1) \\
    & = \frac{s_n+1}{2}a_{n+1} + 2a_{n+1} \\
    & = \frac{1}{2}(s_n+5)a_{n+1} = c a_{n+1}.
\end{align*}
This completes the proof that the equation \eqref{oddeq} is weighted homogeneous of degree $d$ when $n$ is odd.
\end{proof}

Next, we'll prove that the hypersurface $X$ is well-formed, normal,
and Fano with $K_X = \mathcal{O}_X(-1)$.
The notations are as in \Cref{thm:hypersurface} and 
\Cref{thm:hypersurface-odd}.

\begin{lemma}
\label{well-formed}
The weighted projective space $\mathbb{P}$ and the hypersurface $X$ are 
well-formed; in fact, we have $\gcd(a_i,a_{n+1}) = 1$ for each
$i = 0,\ldots,n$ and $\gcd(a_i,a_n) = 1$ for each $i = 0,\ldots,n-1$.
Moreover, $X$ is quasismooth away from the coordinate point of 
the last variable $x_{n+1}$ and $K_X = \mathcal{O}_X(-1)$.
\end{lemma}

\begin{proof}
We begin by verifying that the ambient weighted projective spaces $\mathbb{P}$
are well-formed, meaning that
$\gcd(a_0,\ldots,\widehat{a_j},\ldots,a_{n+1}) = 1$ for all $j = 
0,\ldots,n+1$.  We'll actually prove the stronger gcd conditions in the lemma,
which clearly imply well-formedness.  Through the lemma, we use the notation
$r \coloneqq -1 + a_n + a_{n+1} = d/(s_n-1)$.

Because of the factors of $1/2$ and $1/4$ in the definitions, we can first show that $a_n$ and $a_{n+1}$ are odd.  For $n \geq 1$,
we observe by induction that $s_n \equiv 7 \pmod 8$ when $n$ is even 
and that $s_n \equiv 3 \pmod 8$ when $n$ is odd.  Then for $n$ even,
$s_n^2 + s_n - 4 \equiv 4 \pmod 8$, and for $n$ odd, 
$s_n^2 + 3s_n - 6 \equiv 4 \pmod 8$.  It follows that $a_n$ is always odd.
Since $ba_n + a_{n+1} = d$ and $d$ is even, $a_{n+1}$ is odd as well.

Next, we prove $\gcd(r,a_n) = 1$.  Suppose by way of contradiction
that a prime $p$ divides $r$ and $a_n$.  By the above, we may assume
$p$ is odd.  The equation $r = -1 + a_n + a_{n+1}$ yields $a_{n+1} \equiv 1 \pmod p$.
Now we split into two cases.  For $n$ even, 
$2a_{n+1} = (s_n-1)a_n - s_n -1$ yields $2 \equiv -s_n-1 \pmod p$
so $p$ divides $s_n+3$.  
But $4a_n = s_n^2 +s_n - 4 = (s_n+3)(s_n-2)+2 \equiv 2 \pmod p$,
so $0 \equiv 2 \pmod p$, a contradiction.  Similarly, for $n$ odd,
$4a_{n+1} \equiv (s_n-3)a_n-s_n-1$ yields $4 \equiv -s_n - 1 \pmod p$,
so $p$ divides $s_n+5$.  
But $4a_n = s_n^2 + 3s_n - 6 = (s_n+5)(s_n-2)+4 \equiv 4 \pmod p$,
so $4 \equiv 0 \pmod p$, a contradiction.

To show $\gcd(r,a_{n+1}) = 1$, suppose that $p$ odd divides $r$
and $a_{n+1}$.  This gives $a_n \equiv 1 \pmod p$.  For $n$
even, $2a_{n+1} = (s_n-1)a_n - s_n -1 \equiv s_n - 1 - s_n -1 = -2 \pmod p$,
so $0 \equiv -2 \pmod p$, a contradiction.  For $n$ odd, we similarly
get $4a_{n+1} = (s_n-3)a_n - s_n-1 \equiv s_n - 3 - s_n - 1 = -4 \pmod p$
so $0 \equiv -4 \pmod p$, a contradiction.

Next, we calculate $\gcd(s_n-1,a_n) = 1$. Assume an odd prime $p$ divides
both.  When $n$ is even, $4a_n = s_n^2 + s_n - 4 = (s_n-1)(s_n+2) - 2$
shows that $0 \equiv -2 \pmod p$.  When $n$ is odd,
$4a_n = s_n^2 + 3s_n - 6 = (s_n-1)(s_n+4) - 2$ shows $0 \equiv -2 \pmod p$.
This is a contradiction in either case.

Similarly, we can show $\gcd(s_n-1,a_{n+1}) = 1$. Assume an odd prime $p$ divides
both.  When $n$ is even, we use $2a_{n+1} = (s_n-1)(-1+a_n) - 2 \equiv 
-2 \pmod p$, implying $0 \equiv -2 \pmod p$.  When $n$ is odd,
$4a_{n+1} = (s_n-3)a_n - s_n - 1 = (s_n-1)a_n - 2a_n - s_n - 1 \equiv 
-2a_n - 2 \pmod p$.  Hence $p$ divides $2a_n + 2$. But 
$$2a_n + 2 = \frac{1}{2}(s_n^2 + 3s_n - 6) + 2 = \frac{1}{2}(s_n-1)(s_n + 4)
+ 1 \equiv 1 \pmod p.$$
This means $0 \equiv 1 \pmod p$.  In either case, we have a contradiction.

Lastly, we'll show $\gcd(a_n,a_{n+1}) = 1$. Assume $p$ odd divides both.  When $n$
is even, notice that $2a_{n+1} \equiv -s_n-1 \pmod p$ so $p$ divides
$s_n+1$.  But this would mean
$$a_n = \frac{1}{4}(s_n(s_n+1) - 4) \equiv -1 \pmod p,$$
a contradiction.  When $n$ is odd, we also get $p | s_n+1$, but
$$a_n = \frac{1}{4}((s_n+1)(s_n+2) - 8) \equiv -2 \pmod p,$$
so $0 \equiv 2 \pmod p$, a contradiction.

This concludes the proof of all the gcd statements in the lemma. Indeed,
$\gcd(r,a_n) = \gcd(s_n-1,a_n) = 1$ implies $\gcd(a_i,a_n) = 1$ for
each $i = 0,\ldots,n-1$, and similarly for $a_{n+1}$.  In particular,
$\mathbb{P}$ is well-formed.

The hypersurface $X$ is also well-formed because $\mathbb{P}$ is
well-formed, and the presence of the monomials $x_0^2,\ldots,x_{n-1}^{s_{n-1}},
x_n^b x_{n+1}$ means that only the coordinate points of 
$x_n$ and $x_{n+1}$ are contained in $X$ among all the toric strata of $X$.
Hence the hypersurface $X$ does not contain any stratum of 
$\mathrm{Sing}(\mathbb{P})$ of codimension $2$ in $\mathbb{P}$.
Each of these coordinate points
point has codimension at least $2$ in $X$ since $\dim(X) \geq 2$,
so $X$ is well-formed.

In light of \Cref{rem:coefficients}, we can
take $X$ to be general in the linear system
defined by the monomials of the defining equation.
The only base points for this linear system on $\mathbb{A}^{n+2}$
are multiples of $(0,\ldots,1,0)$ and $(0,\ldots,0,1)$. 
Hence $X$
is quasismooth away the coordinate points of the last two 
coordinates.  Since $x_n^b x_{n+1}$ is a monomial in the 
equation for $X$, a general equation of the linear system
will have a non-vanishing derivative at $(0,\ldots,1,0)$,
so $X$ is also quasismooth at the coordinate point of $x_n$.
Hence only the coordinate point of $x_{n+1}$ is a non-quasismooth
point.
Since $X$ is a hypersurface and the non-quasismooth
locus has codimension at least $2$ in $X$,
it follows that $X$ is normal.
Indeed, we can verify the normality of $X$
in each affine chart $\{x_i \neq 0\} = \mathbb{A}^{n+1}/\mu_{a_i}$
of $\mathbb{P}$ by Serre's criterion:
hypersurfaces in $\mathbb{A}^{n+1}$ are Cohen-Macaulay,
the singular locus of the defining equation
in this affine space has codimension at least
$2$, and the quotient by $\mu_{a_i}$ preserves normality.

We've shown that $X$ is normal and 
well-formed.  Therefore, the adjunction formula holds
for $X$ (cf. \cite[p. 3]{Totaro23}) and $K_X = 
\mathcal{O}_X(d - a_0 - \cdots - a_{n+1}) = \mathcal{O}_X(-1)$.
\end{proof}

\section{Computation of the minimal log discrepancy}\label{sec: compute mld}

In this section, we compute the exact value of $\mld(X)$ in each
dimension, where $X$ is the hypersurface from Theorems \ref{thm:hypersurface}
and \ref{thm:hypersurface-odd}.  We continue to use all notation from
\Cref{sec: well-formed}.  The value of the mld comes from
computing the difference $\beta(x_0 \cdots x_n) - \beta(f)$
from \Cref{thm:mldcomp} for a particular hypersurface and 
lattice point $\beta$.  Much of the length of the proof 
is then spent confirming that this is 
actually the smallest log discrepancy achieved
by any divisor over $X$.

\begin{theorem}
\label{hypersurfmld}
The minimal log discrepancy of the hypersurface $X$ is
$$\mld(X) = \begin{cases}
    \frac{s_n-1}{2a_{n+1}} & n \text{ is even}, \\
    \frac{s_n-3}{4a_{n+1}} & n \text{ is odd}.
\end{cases}.$$
This mld is achieved at the unique non-quasismooth point of $X$.
\end{theorem}

\begin{proof}

For every $n$, we know from \Cref{well-formed} that
the hypersurface $X$ is normal and that it is
quasismooth away from the 
coordinate point of the last coordinate, $x_{n+1}$.
Away from this point, $X$ has only cyclic quotient 
singularities and we can use \Cref{prop:toricmld} to
compute the mld. At the non-quasismooth point, we'll apply
\Cref{thm:mldcomp} to compute the mld; this will turn out
to be the global minimum.

In the affine chart where $x_{n+1} \neq 0$, we may take $x_{n+1} = 1$,
and $X$ is locally a quotient of the hypersurface in affine space 
$\mathbb{A}^{n+1}_{x_0,\ldots,x_n}$ defined by the equation
$$f \coloneqq x_0^2 + x_1^3 + \cdots + x_{n-1}^{s_{n-1}} + x_n^{b} +
x_1 \cdots x_{n-1} x_n = 0$$
if $n$ is even, and
$$f \coloneqq x_0^2 + x_1^3 + \cdots + x_{n-1}^{s_{n-1}} + x_n^{b} +
x_1 \cdots x_{n-1} x_n^2 = 0$$
if $n$ is odd.  The quotient is by $\mu_{a_{n+1}}$, acting
by $\zeta \cdot (x_0,\ldots,x_{n}) = (\zeta^{a_0}x_0,\ldots,\zeta^{a_n}x_n)$
for $\zeta \in \mu_{a_{n+1}}$.
We already know $\mathbb{P}$ and $X$ are well-formed and normal,
so this hypersurface
of affine space satisfies the normality and well-formedness conditions
of \Cref{thm:mldcomp}.  Furthermore, by \Cref{rem:coefficients},
we can take the coefficients of 
the monomials in the example to be general, so that Newton non-degeneracy
is also automatic
(the base locus of any subset of the monomials in the
equation is contained in the complement of the open torus orbit).

Therefore, \Cref{thm:mldcomp} applies, and we can compute the mld 
of $X$ in this affine chart
by finding
the minimum of $\beta(x_0 \cdots x_n) - \beta(f)$ as $\beta$ varies over 
nonzero points of the lattice $N = \Z^{n+1} + \Z \cdot \frac{1}{a_{n+1}}
(a_0,\ldots,a_n)$ which are contained in the positive orthant (minus the
origin).
It's not difficult to show that this difference is a 
positive integer on nonzero integral points, so we can reduce to 
looking at points
$\beta_j \coloneqq (\beta_j^0,\ldots,\beta_j^n)$ where 
$$\beta_j^i = \left\{\frac{j a_i}{a_{n+1}}\right\}.$$
for each $i = 0,\ldots,n$.  (Here $\{x\}$ denotes the fractional part
$x - \lfloor x \rfloor$.)
and $j$ is a positive integer satisfying $1 \leq j \leq a_{n+1}-1$.
These points of $N$ are all in the interior of the cube $[0,1]^{n+1}$;
this is because $\gcd(a_i,a_{n+1}) = 1$
for each $i = 0, \ldots, n$, as shown in \Cref{well-formed}.  
Points in the interior correspond to exceptional divisors which 
have center equal to the non-quasismooth point, so the mld
we will find in the proof occurs at that point.  This is consistent
with the fact that the non-quasismooth point is an isolated 
singularity of $X$.
For the rest of the proof, we separately consider the cases 
of when $n$ is even and when $n$ is odd.

\noindent \textbf{Case 1:} $n$ is even.

Since the equation of $X$ is weighted
homogeneous of degree $d$ for the weights $a_0,\ldots,a_{n+1}$, each 
monomial $m$ of the polynomial $f$ above has degree $d$ modulo $a_{n+1}$
with the same weights.  This implies that for each monomial $m$ with 
nonzero coefficient in $f$, $\beta_j(m) \equiv jd/a_{n+1} \pmod \Z$.
However,
$$d = (-1 + a_n + a_{n+1})(s_n-1) \equiv (-1 + a_n)(s_n-1) 
= 2a_{n+1} + 2 \equiv 2 \pmod{a_{n+1}}.$$
It follows that $\beta_j(m) \equiv 2j/a_{n+1} \pmod \Z$.
We note that for $m = x_0^2$, $\beta_j(m)$ is actually equal to 
$2j/a_{n+1}$ for any $j = 1,\ldots,a_{n+1}-1$.  Indeed, the same 
reasoning as above gives that $a_0 = d/2 \equiv 1 \pmod{a_{n+1}}$
so $\beta_j^0 = j/a_{n+1}$.  Then $\beta_j(x_0^2)$ is twice this value.

Similarly, $a_0 + \cdots + a_n \equiv d + 1 \equiv 3 \pmod{a_{n+1}}$, so 
$\beta_j(x_0 \cdots x_n)$ is congruent to $3j/a_{n+1}$ modulo $\Z$.
The difference $\beta_j(x_0 \cdots x_n) - \beta_j(f)$ that we're interested in
must therefore have fractional part 
$$\frac{3j}{a_{n+1}} - \frac{2j}{a_{n+1}} = \frac{j}{a_{n+1}}.$$
Because $\beta(x_0 \cdots x_n) - \beta(x_1 \cdots x_n) = \beta^0 \geq 0$
for any $\beta$ in the positive orthant, we furthermore know that the difference is positive.
Hence $\beta_j(x_0 \cdots x_n) - \beta_j(f)$ is bounded below by $j/a_{n+1}$ 
for each $j$.
In particular, $X$ is klt.

This lower bound is achieved when $j = j_0 \coloneqq (s_n-1)/2$.  Indeed,
we already know
$\beta_{j_0}(x_0^2) = 2j_0/a_{n+1} = s_{n-1}/a_{n+1}$,
and since
$2j_0 < a_{n+1}$, this is a fraction between $0$ and $1$.  Thus, it actually equals 
the minimum $\beta_{j_0}(f)$. Next we verify 
$\beta_{j_0}(x_0 \cdots x_n) = 3j_0/a_{n+1} = (3s_n-3)/(2a_{n+1})$.
For $i = 0,\ldots,n-1$:
$$j_0 a_i = \frac{s_n-1}{s_i} \frac{d}{2} \equiv \frac{s_n-1}{s_i} \pmod{a_{n+1}},$$
so
$$\beta^i_{j_0} = \frac{s_n-1}{s_i a_{n+1}}.$$
For $i = n$, we have
$$j_0 a_n = \frac{1}{2}((s_n-1)a_n) = a_{n+1} + \frac{1}{2}(s_n+1) \equiv \frac{1}{2}(s_n+1) \pmod{a_{n+1}}.$$
This proves 
$$\beta^n_{j_0} = \frac{s_n+1}{2a_{n+1}}.$$
Adding these together yields
\begin{align*}
\beta_{j_0}(x_0 \cdots x_n) & = \frac{s_n-1}{s_0 a_{n+1}} + \cdots + \frac{s_n-1}{s_{n-1} a_{n+1}} + \frac{s_n+1}{2a_{n+1}} \\
& = \frac{(s_n-1)(s_n-2)}{(s_n-1)a_{n+1}} + \frac{s_n+1}{2a_{n+1}} \\
& = \frac{3s_n-3}{2a_{n+1}},
\end{align*}
as claimed. Subtracting,
$$\beta_{j_0}(x_0 \cdots x_n) - \beta_{j_0}(f) = \frac{s_{n}-1}{2a_{n+1}},$$
proving that there is a divisor over $X$
with log discrepancy equal to $(s_{n}-1)/(2a_{n+1})$.

It remains to see that this is actually the minimum value
of $\beta_j(x_0 \cdots x_n) - \beta_j(f)$ for any $j = 1,\ldots,a_{n+1}-1$. 
In light of the lower bound $j/a_{n+1}$ for this difference,
we need only verify that the difference is larger for $1 \leq j < j_0$.
For such a $j$,
\begin{align*}
\beta_j^0 + \cdots + \beta_j^{n-1} & \geq  \left\{\frac{1}{a_{n+1}}\left(\frac{dj}{s_0} + \cdots + \frac{dj}{s_{n-1}} \right) \right\} \\ 
& =\left\{\frac{1}{a_{n+1}}\frac{(s_n -2)dj}{s_n-1} \right\} \\ 
& = \left\{\frac{1}{a_{n+1}}(s_n-2)(-1 + a_n + a_{n+1})j \right\}.
\end{align*}
Modulo $a_{n+1}$ we have
\begin{align*}
(s_n-2)(-1 + a_n + a_{n+1}) & \equiv (s_n-2)(-1 + a_n) = (s_n-1)(-1+a_n) - (-1 + a_n) \\
& = 2a_{n+1} + 2 - (-1 + a_n) \equiv 3 - a_n \pmod{a_{n+1}}.
\end{align*}

Since $a_{n+1} > a_{n+1} + (3-a_n)j > 3j$ for $1 \leq j < j_0$ by definition of
$a_{n+1}$, 
$$\beta_j(x_0 \cdots x_n) = (\beta_j^0 + \cdots + \beta_j^{n-1}) + \beta_j^n 
\geq \frac{a_{n+1} + (3-a_n)j}{a_{n+1}} + \beta_j^n > \frac{3j}{a_{n+1}}.$$
Since $\beta_j(x_0 \cdots x_n)$ is equivalent to $3j/a_{n+1}$ modulo $\Z$, 
it follows that $\beta_j(x_0 \cdots x_n) \geq 1 + 3j/a_{n+1}$
for $1 \leq j < j_0$.  Therefore,
the difference $\beta_j(x_0 \cdots x_n) - \beta_j(f)$
is greater than $1$ for this range of $j$ values.
This completes the argument.

\noindent \textbf{Case 2:} $n$ is odd.

The argument for odd $n$ is very similar to the argument for even $n$.  We note that
\begin{align*}
    d & = (-1 + a_n + a_{n+1})(s_n-1) \equiv (-1 + a_n)(s_n-3) + 2(-1 + a_n) \\ 
  & \equiv 4a_{n+1} + 4 + 2(-1 + a_n) \equiv 2 + 2a_n \pmod{a_{n+1}}.
\end{align*}
For each monomial $m$ in $f$, we have that $\beta_j(m)$ is
equivalent to $(2 + 2 a_n)/a_{n+1}$ modulo $\Z$.  The sum of the 
weights is still $d + 1$, so it's again true that
$\beta_j(x_0 \cdots x_n) - \beta_j(f) \equiv j/a_{n+1} \pmod \Z$.
For any $j$, 
$$\beta_j(x_0 \cdots x_n) - \beta_j(f) \geq \beta_j(x_0 \cdots x_n) - \beta(x_0^2) = \beta^1_j + \cdots + \beta^n_j - \beta^0_j.$$  
But it's also true that
$$\beta_j(x_0 \cdots x_n) - \beta_j(f) \geq \beta(x_0 \cdots x_n) - \beta(x_1 \cdots x_{n-1} x_n^2) = \beta^0_j - \beta^n_j.$$
Since all the $\beta_j^i$ are positive for $1 \leq j \leq a_{n+1}-1$
(each point is in the interior of the unit cube), one of these two differences is 
positive.  It follows that $X$ is klt.

Next, we confirm that the claimed mld value is attained when $j = j_0 \coloneqq (s_n-3)/4$.  First, for $i = 0, \ldots, n-1$,
\begin{align*}
    j_0 a_i & = \frac{s_n-3}{4}\frac{d}{s_i} = \frac{(s_n-3)(s_n-1)(-1 + a_n + a_{n+1})}{4s_i} \equiv \frac{(s_n-3)(-1 + a_n)(s_n-1)}{4s_i} \\
    & = \frac{(4a_{n+1} + 4)(s_n-1)}{4s_i} \equiv \frac{s_n-1}{s_i} \pmod{a_{n+1}}.
\end{align*}
This shows 
$$\beta^i_{j_0} = \frac{s_n-1}{s_i a_{n+1}}.$$
Similarly, for $i = n$,
$$j_0 a_n = \frac{1}{4}(s_n-3)a_n = a_{n+1} + \frac{1}{4}(s_n+1) \equiv \frac{1}{4}(s_n+1) \pmod{a_{n+1}}.$$
This proves
$$\beta^n_{j_0} = \frac{s_n+1}{4a_{n+1}}.$$
Now we can compute
\begin{align*}
    \beta_{j_0}(x_0 \cdots x_n) & = \frac{s_n-1}{s_0 a_{n+1}} + \cdots + \frac{s_n-1}{s_{n-1}a_{n+1}} + \frac{s_n+1}{4 a_{n+1}} \\
    & = \frac{(s_n-1)(s_n-2)}{(s_n-1)a_{n+1}} + \frac{s_n+1}{4a_{n+1}} \\
    & = \frac{5s_n - 7}{4 a_{n+1}}.
\end{align*}
Similarly,
\begin{align*}
    \beta_{j_0}(x_1 \cdots x_{n-1} x_n^2) & = \frac{s_n-1}{s_1 a_{n+1}} + \cdots + \frac{s_n-1}{s_{n-1}a_{n+1}} + \frac{s_n + 1}{2 a_{n+1}} \\
    & = \frac{(s_n-3)(s_n-1)}{2(s_n-1)a_{n+1}} + \frac{s_n+1}{2a_{n+1}} \\
    & = \frac{s_n-1}{a_{n+1}}.
\end{align*}
Subtracting these gives 
$$\beta_{j_0}(x_0 \cdots x_n) - \beta_{j_0}(x_1 \cdots x_{n-1}x_n^2) = \frac{s_n-3}{4a_{n+1}}.$$
We already know that this is the correct value of the log discrepancy
modulo $\Z$ and this value is between $0$ and $1$,
so in fact $\beta_{j_0}(x_0 \cdots x_n) - \beta_{j_0}(f) = 
(s_n-3)/(4 a_{n+1})$.  It remains to check that 
this is in fact the lowest value attained for any $j$.
As in the even case, we only need to verify that no lower value is attained
in the range $1 \leq j < j_0$.  To see this, first note that, 
$$\beta_j(x_0 \cdots x_n) - \beta_j(f) \geq \beta_j(x_0 \cdots x_n)
- \beta_j(x_0^2) = \beta_j^n - \beta_j^0 + \beta_j^1 + \cdots + \beta_j^{n-1}.$$
Since $\beta_j^0 - \beta_j^n$ is also of the form 
$\beta_j(x_0 \cdots x_n) - \beta_j(m)$ for the monomial 
$m = x_1 \cdots x_{n-1} x_n^2$ of $f$, this difference is
equivalent to $j/a_{n+1}$ modulo $\Z$.
Since both $\beta_j^0$ and $\beta_j^n$ are between $0$ and $1$, 
it's simple to see that $\beta_j^n - \beta_j^0 \geq - j/a_{n+1}$
(and again is equivalent to it modulo $\Z$).

Therefore, the only way that the difference $\beta_j(x_0 \cdots x_n) - \beta_j(x_0^2)$
could achieve the minimum possible value $j/a_{n+1}$ is
if $\beta_j^1 + \cdots + \beta_j^{n-1} = 2j/a_{n+1}$.  
Suppose this holds, for some particular $j$.  This would imply that
$$\beta_j(x_1^3) = 3 \beta_j^1 \leq 3\frac{2j}{a_{n+1}} = \frac{6j}{a_{n+1}}.$$
In comparison, we can calculate
$$\beta_j^0 = \frac{j(1+a_n)}{a_{n+1}}$$
for any $j$ in this range we're considering, because 
$a_0 = d/2 \equiv 1 + a_n \pmod{a_{n+1}}$ and $j(1+a_n) < a_{n+1}$
when $1 \leq j < j_0$. But then
$$\beta_j(x_0 \cdots x_n) - \beta_j(f) \geq \beta_j(x_0 \cdots x_n) - 
\beta_j(x_1^3) \geq \frac{j(1+a_n)}{a_{n+1}} - \frac{6j}{a_{n+1}}
= \frac{(a_n-5)j}{a_{n+1}}.$$
This last expression is greater than $j/a_{n+1}$ because $a_n > 6$.
Since the discrepancy is equivalent to $j/a_{n+1}$ modulo $\Z$, it
must be greater than $1$ for the given value of $j$.

So far, we've demonstrated that $\mld(X,x)$ equals the value
from \Cref{hypersurfmld} when $x$ is the unique non-quasismooth point of $X$.
However, we also need to check that no smaller mld is attained at any
quasismooth point of the variety $X$.  The following lemma finishes the proof.

\begin{lemma}
\label{mld_otherpts}
Let $x \in X$ be the non-quasismooth point of the hypersurface $X$ defined
above.  Then $\mld(X \setminus x) \geq 2/a_n$.
\end{lemma}

\begin{proof}[Proof of \Cref{mld_otherpts}:]
We can take hypersurface $X$ to be a general
member of a linear series and it is quasismooth away from $x$
so the singularities on all toric strata apart from $x$
are the quotient singularities
described in, e.g., \cite[Proposition 2.6]{ETW}.  In particular, the 
singularity type is constant along the intersection of $X$ with any
toric stratum of $\mathbb{P}$, and we need only check the mld
of the singularities on $1$-dimensional toric strata in $\mathbb{P}$
(when neither coordinate point in its closure is in $X$)
or at coordinate points of $\mathbb{P}$ contained in $X$.
This is in the sense of \Cref{prop:toricmld}: every quotient
singularity on a larger stratum of $X$ already appears
somewhere on the affine toric variety corresponding to the
singularities on smaller strata.
In particular, we look at quotient singularities at
points $y$ in $X$ such that:
1) the only nonzero coordinates of $y$ are $x_{i_1}$ and $x_{i_2}$
for some $i_1,i_2$ among $0,\ldots,n-1$, or 
2) $y$ is the coordinate point of $x_n$.

In case 1), the singularity type is always of type $\frac{1}{mr}(\ldots,a_n,a_{n+1})$,
where $r = -1 + a_n + a_{n+1}$ and $mr$ the gcd of two of the weights among $a_0,\ldots,a_{n-1}$. 
The multiple $m$ here divides $s_n-1$, but is at most $(s_n-1)/6$. We
omit the other weights in the singularity, because it will actually 
be enough to consider the contributions of the last two in the 
formula of \Cref{prop:toricmld}.  The sum of
the weights in the singularity must equal $d + 1 \equiv 1 \pmod{mr}$,
so the sum in \eqref{toricmldeq} is equivalent to $j/(mr)$
modulo $\Z$ for each $j$.

When $j < mr/a_n$, $ja_n < mr$, so that
$$\frac{ja_n}{mr} - \bigg\lfloor \frac{ja_n}{mr} \bigg\rfloor = \frac{ja_n}{mr} > \frac{j}{mr}.$$
The left hand side is just one term in in the sum \eqref{toricmldeq},
so the total sum is greater than
$1$ for this $j$.  The first multiple of $a_n$ which is greater than $mr$
occurs when $j = m(s_n+1)/2$, for $n$ even, and when 
$j = m(s_n+1)/4$, for $n$ odd. Indeed, for the even case,
\begin{align*}
    \frac{1}{2}m(s_n+1)a_n - mr & = m\left(\frac{1}{2}(s_n+1)a_n + 1 - a_n - \frac{1}{2}((s_n-1)a_n - s_n - 1)\right) = \frac{1}{2}m(s_n + 3),
\end{align*}
which is between $0$ and $a_n$, since $m \leq (s_n-1)/6$.  Similarly,
for the odd case,
\begin{align*}
    \frac{1}{4}m(s_n+1)a_n - mr & = m\left(\frac{1}{4}(s_n+1)a_n + 1 - a_n - \frac{1}{4}((s_n-3)a_n - s_n - 1)\right) = \frac{1}{4}m(s_n + 5),
\end{align*}
which has the same property.  In both equations above, the right-hand side is 
the numerator of the difference $ja_n/(mr) - \big\lfloor ja_n/(mr) \big\rfloor$,
for the indicated special value of $j$.  Note that it is still greater than $j$
in both cases, so the sum in \eqref{toricmldeq} is greater than $1$ for this $j$ too.
As $j$ increases further, the contribution from $a_n$
in the sum in \eqref{toricmldeq} will continue to be larger than $j/(mr)$
until, at least, $ja_n$ reaches $2mr$.  
This proves that the minimal log discrepancy of any quotient singularity of this type is
at least
$$\frac{2mr/a_n}{mr} = \frac{2}{a_n},$$
as required.

The remaining singularity from case (2) is even simpler.  At the coordinate
point of $x_n$, the singularity type is $\frac{1}{a_n}(a_0,\ldots,a_{n-1})$.
By \Cref{well-formed}, $\gcd(a_i,a_n) = 1$ for each $i = 0,\ldots,n-1$,
so the sum \eqref{toricmldeq} always has $n$ positive terms for each 
$j = 1,\ldots,a_n-1$.  This already proves that the mld is at least
$n/a_n \geq 2/a_n$.
\end{proof}

\Cref{hypersurfmld} is now proved by noting that $2/a_n$ is larger than the mld
found at the non-quasismooth point for all $n$.
\end{proof}

\section{Computation of the alpha invariant and exceptionality}\label{sec: compute alpha}

In this section, we demonstrate that the Fano varieties
defined in \Cref{thm:hypersurface} and \Cref{thm:hypersurface-odd}
are exceptional.

Let $X$ be a hypersurface in $\mathbb{A}_{\mathbb{C}}^{n+1}$ that contains the origin.
For positive integers $c_0,\ldots,c_n$, we take $f:X' \rightarrow X$ to be the stack-theoretic weighted blow-up of $X\subset \mathbb{A}^{n+1}$ at the origin with the given weights $c=(c_0,\ldots,c_n)$. Section 3.4 in \cite{ATW} gives explicit coordinate charts. We define the \textit{weighted tangent cone} $X^c$ to be the exceptional divisor in $X'$. In particular, we view $X^c$ as a hypersurface in the stack $\mathbb{P}^{n}(c_0,\ldots,c_n)=[(\mathbb{A}^{n+1}-0)/\mathbb{G}_{\on{m}}]$.

For positive integers $c_0,\ldots,c_n$, the \textit{weighted multiplicity} $\mathrm{mult}_c{S}$ of a closed subscheme $S\subset \mathbb{A}^{n+1}$ at the origin is defined to be the degree of the weighted tangent cone $S^c$ as a substack of the weighted projective space  $\mathbb{P}^{n}(c_0,\ldots,c_n)$. Equivalently, let $F^l\mathcal{O}(S)$ be the linear span of the monomials $x^{I}$ with $\sum^n_{j=0}c_ji_j\geq l$. This gives a decreasing filtration of the ring $\mathcal{O}(S)$ of regular functions. Then the weighted multiplicity is the limit 
\[\lim\limits_{l\to \infty}\frac{\mathrm{dim}(\mathcal{O}(S)/F^l\mathcal{O}(S))}{l^m/m!},\]
where $m$ is the dimension of $S$. If all weights are $1$, this limit equals the usual multiplicity of $S$ at the origin \cite[Example 4.3.1]{Fulton}.

\begin{lemma}\label{even-klt-Newton}
For each even integer $n\geq 2$, the hypersurface $S$ in 
$\mathbb{A}^{n+1}$ defined by the equation 
$$x_0^2+x_1^3+\cdots+x_{n-1}^{s_{n-1}}+x_1\cdots x_{n}=0$$ 
is canonical, and in particular klt.
\end{lemma}

\begin{proof}
The number of monomials in this equation is the same as the number
of variables, so by the same reasoning as in \Cref{rem:coefficients},
we may actually
replace $S$ with a hypersurface defined by a general linear
combination of the monomials in the equation above (this new
hypersurface is isomorphic to $S$).
The base locus of the linear system generated by the monomials is
the line spanned by $(0,\ldots,0,1)$, so by
Bertini's theorem, the surface $S$ is smooth in codimension $1$. 
It is also Cohen-Macaulay since it is a hypersurface
in affine space, so it is normal by Serre's criterion.

Now we show that the convex hull of the points 
$(2,0,\ldots,0),(0,3,0,\ldots,0),\ldots,(0,\ldots,0,s_{n-1},0)$ 
and $(0,1,\ldots,1)$ in $\mathbb{R}^{n+1}$ contains a point with
all coordinates less than $1$. Indeed, the point 
$$\frac{5}{12}(2,0,\ldots,0)+\frac{1}{6}(0,3,0,\ldots,0)+\frac{5}{12}(0,1,\ldots,1)$$
has all coordinates less than $1$. 
So the Newton polytope of $x_0^2,x_1^3,\ldots,x^{s_{n-1}}_{n-1}$ and $x_1\cdots x_n$ 
contains $(1,\ldots,1)$ in the interior.
\Cref{cor:klt_hypersurf} implies that $S$ is klt.  Since $S$ is
Gorenstein, it is canonical.
\end{proof}

\begin{lemma}\label{odd-klt-Newton}
For each odd integer $n\geq 3$, the hypersurface 
$S$ in $\mathbb{A}^{n+1}$ defined by the equation 
$$x_0^2+x_1^3+\cdots+x_{n-1}^{s_{n-1}}+x_1\cdots x^2_{n}=0$$ is canonical, in particular klt.
\end{lemma}
\begin{proof}
The same argument from \Cref{even-klt-Newton} shows that 
$S$ is normal and that we may replace $S$ by the hypersurface
defined by the same equation but with general coefficients.

We show that the convex hull of the points
$(2,0,\ldots,0),(0,3,0,\ldots,0),\ldots,(0,\ldots,0,s_{n-1},0)$ and $(0,1,\ldots,1,2)$ 
in $\mathbb{R}^{n+1}$ contains a point with all coordinates less than $1$. 
Indeed, the point 
$$\frac{5}{12}(2,0,\ldots,0)+\frac{1}{6}(0,3,0,\ldots,0)+\frac{5}{12}(0,1,\ldots,1,2)$$ 
has all coordinates less than $1$. 
So the Newton polytope of $x_0^2,x_1^3,\ldots,x^{s_{n-1}}_{n-1}$ and $x_1\cdots x^2_n$
contains $(1,\ldots,1)$ in the interior. 
\Cref{cor:klt_hypersurf} implies that $S$ is klt, in particular canonical.
\end{proof}

In weighted projective space $\mathbb{P}=\mathbb{P}(a_0,\ldots,a_n)$, the multiplicity of an irreducible closed substack (or an effective algebraic cycle) at a point is defined to be the multiplicity at a corresponding point of its inverse image in any orbifold chart $\mathbb{A}^n\rightarrow [\mathbb{A}^n/ \mu_{a_i}]\cong\{x_i\neq 0\}\subset \mathbb{P}$. It is independent of the choice of $i$ since different orbifold charts are \'etale-locally isomorphic.
We can use the following criterion to check that a pair is canonical using multiplicity.

\begin{theorem}{{{\cite[Claim 2.10.4]{Kollar2013}}}}
\label{claim2.10.4}
    Let $X$ be a regular scheme and $\Delta$ an effective $\mathbb{Q}$-divisor. Then $(X,\Delta)$ is canonical if $\mathrm{mult}_p\Delta\leq 1$ for every point $p$ in $X$.
\end{theorem}

Johnson and Koll\'ar gave the following bound for the multiplicity
of an irreducible closed substack of weighted projective space.

\begin{lemma}{{{\cite[Proposition 11]{JohnsonKollar2001}}}}
\label{multi-bound}
Let $a_0\geq \ldots \geq a_{n}$ be positive integers and 
$T$ be an irreducible closed substack of $\mathbb{P}^n(a_0,\ldots,a_n)$
of dimension $m<n$.
Then the multiplicity of $T$ at every point is at most $(a_0\cdots a_m)\deg(T)$.
\end{lemma}

Now we prove that the hypersurfaces from 
\Cref{thm:hypersurface} and \Cref{thm:hypersurface-odd}
are exceptional.

\begin{theorem}\label{thm: exceptionality}
    For every even integer $n \geq 2$, the $n$-dimensional klt Fano variety $X$ defined in 
    \Cref{thm:hypersurface} is exceptional. The Fano $n$-fold $X$ has $\alpha$-invariant $$1<\frac{a_n a_{n+1}}{d}\leq \alpha(X)\leq \frac{(s_n-2)a_{n+1}}{s_n-1},$$ where $a_n a_{n+1}/d \sim s_{n}/4$ and $(s_n-2)a_{n+1}/(s_n-1) \sim s_n^3/8$.
\end{theorem}

\begin{proof}
     To show that $X$ is exceptional, it is equivalent to show that the $\alpha$-invariant of $X$ is greater than $1$. We now aim to show that there exists $\nu_n>1$ such that the pair $(X,\nu_nD)$ is lc for every effective $\mathbb{Q}$-divisor $D\sim_{\mathbb{Q}}-K_X$. More precisely, we will show that if $\nu_n \leq a_n a_{n+1}/d$, then $(X,\nu_nD)$ is lc for every effective $\mathbb{Q}$-divisor $D\sim_{\mathbb{Q}}-K_X$.

    Since $D\sim_{\mathbb{Q}}-K_X$ and $-K_X=\mathcal{O}(1)$, we have $\mathrm{deg}(D)=d/(a_0\cdots a_{n+1})$. Note that $a_0\geq \cdots \geq a_{n-1}\geq a_{n+1}\geq a_n$. Then by \Cref{multi-bound}, the multiplicity of $D$ at every point is at most $a_0\cdots a_{n-1}\mathrm{deg}(D)=d/(a_n a_{n+1})$, where $a_0\cdots a_{n-1}$ is the product of the $n$ largest weights. If $\nu_n\leq a_n a_{n+1}/d$, then $\nu_nD$ has multiplicity at every point at most $\nu_n d/(a_n a_{n+1})$ which is at most $1$. By Theorem \ref{claim2.10.4}, when 
\begin{equation}\label{nu_n,multiplicityX<=1}
\nu_n\leq \frac{a_n a_{n+1}}{d},
    \end{equation}
the pair $(X,\nu_nD)$ is lc at all smooth points of the stack $X$, hence at all points other than $[x_0:\cdots:x_n:x_{n+1}]=[0:\cdots:0:1].$ Note that $a_n a_{n+1}/d\sim s_{n}/4$ since $a_n\sim s_n^2/4$, $a_{n+1}\sim s_n a_n/2 \sim s_n^3/8$ and $d\sim a_{n+1}s_n \sim s_{n}^4/8$. In particular, $a_n a_{n+1}/d > 1$ for all even $n \geq 2$.

In the affine chart $x_{n+1}=1$, $X$ is the hypersurface of $\mathbb{A}^{n+1}$ given by $x_0^2+x_1^3+\cdots+x_{n-1}^{s_{n-1}}+x_n^b+x_1\cdots x_n=0$. We will show that $(X,\nu_nD)$ is lc near the origin $(0,\ldots,0)$ for any $\nu_n \leq a_n a_{n+1}/d$. By \cite[Lemma 4.4]{Totaro23}, we get $$\mathrm{mult}_a{D}\leq a_{n+1}\mathrm{deg}(D)=\frac{d}{a_0\cdots a_n},$$
where $\mathrm{mult}_a{D}$ is the $a$-weighted multiplicity of $D$ at the origin in $\mathbb{A}^{n+1}$ with respect to the weights $a_0,\ldots,a_n$ for $x_0,\ldots,x_{n}$, respectively.

Let $r \coloneqq -1 + a_n + a_{n+1}$.
Let $\mathrm{mult}_u{D}$ be the $u$-weighted multiplicity of $D$ at the origin with respect to the weights $u_i \coloneqq a_i$ for $i=0,\ldots,n-1$, and $u_n \coloneqq (s_n+1)r/2$. Since $u_n>a_n$, we have $\mathrm{mult}_u{D}\leq \mathrm{mult}_a(D)\leq d/(a_0\cdots a_n)$. Let $b_i \coloneqq u_i/r$ for $i=0,\ldots,n$. Then $\mathrm{mult}_b{D}=r^{n-1}\mathrm{mult}_u{D}\leq r^{n-1}d/(a_0\cdots a_n)$.
Since $d=(s_n-1)r$ and $a_i=b_ir$ for $i=0,\ldots,n-1$, we get $\mathrm{mult}_b{D}\leq (s_n-1)/(b_0\cdots b_{n-1} a_n)$. 

Note that $b_i=(s_n-1)/s_i$ for $i=0,\ldots,n-1$, and $b_n=(s_n+1)/2$. Let $X^{c}$ be the hypersurface in $\mathbb{P}(b_0,\ldots,b_n)$ of degree $s_n-1$ defined by $x_0^2+x_1^3+\cdots+x_{n-1}^{s_{n-1}}+x_1\cdots x_{n}=0,$ which is the weighted tangent cone to the hypersurface $X\subset \mathbb{A}^{n+1}$ at the origin  with respect to the weights $b_0,\ldots,b_n$. The monomial $x_n^b$ has disappeared
since its degree in the weights $b_0,\ldots,b_n$ is larger than $s_n-1$. Because the klt property is preserved by finite quotients, Lemma \ref{even-klt-Newton} means $X^{c}$ is klt. 
 Denote the weighted tangent cone of $D \subset \mathbb{A}^{n+1}$ at the origin by $D^{c}$. Then $D^{c}$ is an effective $\mathbb{Q}$-divisor in $X^{c}$ and is $\mathbb{Q}$-linearly equivalent to a rational multiple of $\mathcal{O}_{X^c}(1)$ since $D$ is $\mathbb{Q}$-Cartier. By \cite[Lemma 4.1]{Totaro23}, if $(X^c,\nu_nD^c)$ is lc and $K_{X^c}+\nu_n D^c \sim_{\mathbb{Q}}\mathcal{O}_{X^c}(l)$ with $l\leq 0$, then $(X,\nu_nD)$ is lc near the origin.

We have $-K_{X^c}=\mathcal{O}_{X^c}(-(s_n-1)+\sum_{j=0}^nb_j)=\mathcal{O}_{X^c}((s_n-1)/2
)$ by the adjunction formula. Note that $l\leq 0$ is equivalent to 
$\nu_n\mathrm{deg}(D^c)\leq \mathrm{deg}(-K_{X^c})=(-K_{X^c})\cdot c_1(\mathcal{O}(1))^{n-2}=(s_n-1)/2 \cdot (s_n-1)/(b_0\cdots b_n).
$
That is to say, $l\leq 0$ iff 
$$\nu_n\leq\frac{(s_n-1)^2}{2b_0\cdots b_n\mathrm{deg}(D^c)}.$$
The right side of this inequality is at least $(s_n-1)a_n/(2b_n)$
since $\mathrm{deg}(D^c)=\mathrm{mult}_b{D}\leq (s_n-1)/(b_0\cdots b_{n-1} a_n)$.
Thus we have $l\leq 0$ when 
\begin{equation}\label{nu_n,l<=0}
 \nu_n\leq \frac{(s_n-1)a_n}{2b_n}=\frac{(s_n-1)a_n}{s_n+1},
 \end{equation}
where $(s_n-1)a_n/(2b_n) \sim s_n^2/4$ since $a_n\sim s_n^2/4$ and $b_n\sim s_n/2$. 

Now we want to find the inequality satisfied by $\nu_n$ so that $(X^c,\nu_nD^c)$ is lc. Note that $b_n>b_0>\cdots>b_{n-1}$. By \Cref{multi-bound}, since the dimension of $D^c$ is $n-2$, the multiplicity of $D^c$ at every point is at most $b_2\mathrm{deg}(D^c)$ when $n=2$, or $b_nb_0\cdots b_{n-3}\mathrm{deg}(D^c)$ when $n\geq 4$, where $b_nb_0\cdots b_{n-3}$ is the product of the $n-1$ largest weights.
Since $\mathrm{deg}(D^c)\leq (s_n-1)/(b_0\cdots b_{n-1} a_n)$, the multiplicity of $D^c$ at every point is at most $b_2(s_2-1)/(b_0b_1a_2)=4/13$ when $n=2$, or $b_n(s_n-1)/(b_{n-2}b_{n-1}a_n)$ when $n\geq 4$. Therefore, for each even integer $n\geq 2$, the multiplicity of $\nu_nD^c$ at every point is at most $1$ if 
\begin{equation}\label{nu_n,multiplicityD^c<=1}
 \nu_n\leq \frac{b_{n-2}b_{n-1}a_n}{b_n(s_n-1)}=\frac{2a_n(s_n-1)}{s_{n-2}s_{n-1}(s_n+1)}.
 \end{equation}
 
 We have $2a_n(s_n-1)/(s_{n-2}s_{n-1}(s_n+1))\sim s^5_{n-2}/2$ since $s_n-1\sim s_n \sim s_{n-1}^2\sim s_{n-2}^4$ and $a_n\sim s_n^2/4\sim s^8_{n-2}/4$.
Since the stack $X^c$ is a smooth outside the point $[x_0:\cdots :x_{n-1}:x_{n}]=[0:\cdots:0:1]$, by Theorem \ref{claim2.10.4},  $(X^c,\nu_nD^c)$ is lc outside this point if $\nu$ satisfies inequality \eqref{nu_n,multiplicityD^c<=1}.

Now we want to find the inequality satisfied by $\nu_n$ such that $(X^c,\nu_nD^c)$ is lc near the point $[x_0:\cdots:x_{n-1}:x_{n}]=[0:\cdots:0:1]$ in $\mathbb{P}(b_0,\ldots,b_n)$. 
In the affine coordinate chart $x_n=1$, the hypersurface $X^c$ is defined by $x_0^2+x_1^3+\cdots+x_{n-1}^{s_{n-1}}+x_1\cdots x_{n-1}=0$. Note that $b_i=c_i$ for $i=0,\ldots,n-1$, for the values of $c_i$ appearing in \cite[Lemma 5.1]{Totaro23}. Hence the pair $(X^c,\nu_nD^c)$ is lc near the point $[x_0:\cdots:x_{n-1}:x_{n}]=[0:\cdots:0:1]$ if
\begin{equation*}
    \nu_n\leq 
    \begin{dcases}
        \frac{1}{\mathrm{mult}_b{D^c}} \text{ for } n=2,\\
        \frac{2}{s_{n-1}^{n-2}(s_{n-1}+1)^2(s_{n-1}-1)^{n-4}\mathrm{mult}_b{D^c}} \text{ for } n\geq 4.
    \end{dcases}
\end{equation*}
We know $D^c$ corresponds to a cycle with codimension $2$ in this affine coordinate chart $\mathbb{A}^n$ and by \cite[Lemma 4.4]{Totaro23} $$\mathrm{mult}_b{D^c}\leq b_n \mathrm{deg}(D^c)=\frac{b_n(s_n-1)}{b_0\cdots b_{n-1}a_n}=\frac{s_n+1}{2(s_n-1)^{n-2}a_n}.$$
Thus $(X^c,\nu_nD^c)$ is lc near the point $[x_0:\cdots:x_{n-1}:x_{n}]=[0:\cdots:0:1]$ if \begin{equation}\label{nu_n,lcatpoint}
\nu_n\leq 
    \begin{dcases}
        \frac{2a_2}{s_2+1}=\frac{13}{4} \text{ for } n=2,\\
        \frac{4a_n(s_n-1)^{n-2}}{s_{n-1}^{n-2}(s_{n-1}+1)^2(s_{n-1}-1)^{n-4}(s_n+1)}\sim s_{n-1}^2 \text{ for } n\geq 4
    \end{dcases}
 \end{equation}
 since $4a_n(s_n-1)^{n-2}\sim 4\cdot s_{n-1}^4(s_n-1)^{n-2}/4=s_{n-1}^{2n}$ and $s_{n-1}^{n-2}(s_{n-1}+1)^2(s_{n-1}-1)^{n-4}(s_n+1)\sim s_{n-1}^{2n-4}s_{n-1}^2=s_{n-1}^{2n-2}$. 
Among all the right sides of inequalities \eqref{nu_n,multiplicityX<=1}, \eqref{nu_n,l<=0}, \eqref{nu_n,multiplicityD^c<=1} and \eqref{nu_n,lcatpoint}, the one in inequality \eqref{nu_n,multiplicityX<=1} is the smallest one and greater than $1$.  
We know $\alpha$-invariant of $X$ is the supremum of the real numbers $\nu_n$ such that $(X,\nu_nD)$ is lc for every effective $\mathbb{Q}-$divisor $D\sim_{\mathbb{Q}}-K_X$. Hence the $\alpha$-invariant  $\alpha(X)\geq a_n a_{n+1}/d>1$.

Now we show $\alpha(X)\leq (s_n-2)a_{n+1}/(s_n-1)$.
Indeed, there exists a $\mathbb{Q}$-divisor $D\in |-K_X|_{\mathbb{Q}}$ such that $\mathrm{lct}(X,D)=(s_n-2)a_{n+1}/(s_n-1)$. 
Let $D \coloneqq (1/a_{n+1})E$, where $E$ is the hyperplane section in $X$ defined by the equations $\{x_{n+1}=0, x_0^2+\cdots+x_{n-1}^{s_{n-1}}=0\}$
in $\mathbb{P}(a_0,\ldots,a_n)$. The stack $E$ has one singular point $\p_n \coloneqq [x_0:\cdots:x_{n-1}:x_n:x_{n+1}]=[0:\cdots:0:1:0]$, which is a smooth point of stack $X$. In the affine chart $x_n=1$, the stack $E$ is \'etale-locally isomorphic to the Fermat-type hypersurface singularity $x_0^2+\cdots+x_{n-1}^{s_{n-1}}=0$ in $\mathbb{A}^n$. By \cite[Example 8.15]{Kollar97}, the log canonical threshold of the pair $(X,E)$ at the point $\p_n$ is $\mathrm{min}\{1/2+\cdots+1/s_{n-1},1\}=(s_n-2)/(s_n-1)$. Hence the log canonical threshold of $(X, D)$ at $\p_n$ is $(s_n-2)a_{n+1}/(s_n-1)$. Since $\p_n$ is the only singular point in $E$, we get $\mathrm{lct}(X,D) =(s_n-2)a_{n+1}/(s_n-1)$. So $\alpha(X)\leq \mathrm{lct}(X,D) = (s_n-2)a_{n+1}/(s_n-1)$. We have $(s_n-2)a_{n+1}/(s_n-1) \sim s_n^3/8$ since $a_{n+1}\sim s_n a_n/2 \sim s_n^3/8$.
\end{proof}

\begin{theorem}\label{thm: exceptionality-odd}
    For every odd integer $n \geq 3$, the $n$-dimensional klt Fano variety $X$ defined in 
    \Cref{thm:hypersurface-odd} is exceptional. The Fano $n$-fold $X$ has $\alpha$-invariant $$1<\frac{a_n a_{n+1}}{d}\leq \alpha(X)\leq \frac{(s_n-2)a_{n+1}}{s_n-1},$$ where $a_n a_{n+1}/d\sim s_n/4$ and $(s_n-2)a_{n+1}/(s_n-1) \sim s_n^3/16$.
\end{theorem}

\begin{proof}
To show $X$ is exceptional, it is equivalent to show that the $\alpha$-invariant of $X$ is greater than $1$. We now aim to show that there exists $\nu_n>1$ such that the pair $(X,\nu_nD)$ is lc for every effective $\mathbb{Q}$-divisor $D\sim_{\mathbb{Q}}-K_X$. In particular, we can show that if $\nu_n \leq a_n a_{n+1}/d$, then $(X,\nu_nD)$ is lc for every effective $\mathbb{Q}$-divisor $D\sim_{\mathbb{Q}}-K_X$.

    Since $D\sim_{\mathbb{Q}}-K_X$ and $-K_X=\mathcal{O}(1)$, we have $\mathrm{deg}(D)=d/(a_0\cdots a_{n+1})$. Note that $a_0\geq \cdots \geq a_{n-1}\geq a_{n+1}\geq a_n$. Then by \Cref{multi-bound}, the multiplicity of $D$ at every point is at most $a_0\cdots a_{n-1}\mathrm{deg}(D)=d/(a_n a_{n+1})$, where $a_0\cdots a_{n-1}$ is the product of the $n$ largest weights. If $\nu_n\leq a_n a_{n+1}/d$, then $\nu_nD$ has multiplicity at every point at most $\nu_n d/(a_n a_{n+1})$ which is at most $1$. By Theorem \ref{claim2.10.4}, when 
\begin{equation}\label{nu_n,multiplicityX<=1,odd}
\nu_n\leq \frac{a_n a_{n+1}}{d},
    \end{equation}
the pair $(X,\nu_nD)$ is lc at all smooth points of the stack $X$, hence at all points other than $[x_0:\cdots:x_n:x_{n+1}]=[0:\cdots:0:1].$ Note that $(a_n a_{n+1})/d\sim s_n/4$ since $a_n\sim s_n^2/4$, $a_{n+1}\sim (s_n a_n)/4 \sim s_n^3/16$ and $d\sim a_{n+1}s_n \sim s_{n}^4/16$. In particular, $a_n a_{n+1}/d > 1$ for all odd integers $n \geq 3$.

In the affine chart $x_{n+1}=1$, the hypersurface $X\subset \mathbb{A}^{n+1}$ is given by $x_0^2+x_1^3+\cdots+x_{n-1}^{s_{n-1}}+x_n^b+x_1\cdots x_n^2=0$. We will show that $(X,\nu_nD)$ is lc near the origin $(0,\ldots,0)$ for any $\nu_n \leq a_n a_{n+1}/d$. By \cite[Lemma 4.4]{Totaro23}, we get $\mathrm{mult}_a{D}\leq a_{n+1}\mathrm{deg}(D)=d/(a_0\cdots a_n)$,
where $\mathrm{mult}_a{D}$ is the $a$-weighted multiplicity of $D$ at the origin in $\mathbb{A}^{n+1}$ with respect to the weights $a_0,\ldots,a_n$ for $x_0,\ldots,x_{n}$ respectively.

Let $r \coloneqq -1 + a_n + a_{n+1}$.
Let $\mathrm{mult}_u{D}$ be the $u$-weighted multiplicity of $D$ at the origin with respect to the weights $u_i \coloneqq a_i$ for $i=0,\ldots,n-1$, and $u_n \coloneqq (s_n+1)r/4$. Since $u_n>a_n$, we have $\mathrm{mult}_u{D}\leq \mathrm{mult}_a(D)\leq d/(a_0\cdots a_n)$. Let $b_i \coloneqq u_i/r$ for $i=0,\ldots,n$. Then $\mathrm{mult}_b{D}=r^{n-1}\mathrm{mult}_u{D}\leq r^{n-1}d/(a_0\cdots a_n)$.
Since $d=(s_n-1)r$ and $a_i=b_ir$ for $i=0,\ldots,n-1$, we get $\mathrm{mult}_b{D}\leq (s_n-1)/(b_0\cdots b_{n-1} a_n)$.

Note that $b_i=(s_n-1)/s_i$ for $i=0,\ldots,n-1$, and $b_n=(s_n+1)/4$. Let $X^{c}$ be the hypersurface in $\mathbb{P}(b_0,\ldots,b_n)$ of degree $s_n-1$ defined by $x_0^2+x_1^3+\cdots+x_{n-1}^{s_{n-1}}+x_1\cdots x_{n-1}x^2_{n}=0,$ which is the weighted tangent cone to the hypersurface $X\subset \mathbb{A}^{n+1}$ at the origin  with respect to the weights $b_0,\ldots,b_n$. The monomial $x_n^b$ has disappeared
since its degree in the weights $b_0,\ldots,b_n$ is larger than $s_n-1$. Because the klt property is preserved by finite quotients, Lemma \ref{odd-klt-Newton} means $X^{c}$ is klt. 
Denote the weighted tangent cone of $D\in \mathbb{A}^{n+1}$ at the origin by $D^{c}$. Then $D^{c}$ is an effective $\mathbb{Q}$-divisor in $X^{c}$ and is $\mathbb{Q}$-linearly equivalent to a rational multiple of $\mathcal{O}_{X^c}(1)$ since $D$ is $\mathbb{Q}$-Cartier. By \cite[Lemma 4.1]{Totaro23}, if $(X^c,\nu_nD^c)$ is lc and $K_{X^c}+\nu_n D^c \sim_{\mathbb{Q}}\mathcal{O}_{X^c}(l)$ with $l\leq 0$, then $(X,\nu_nD)$ is lc near the origin.

We have $-K_{X^c}=\mathcal{O}_{X^c}(-(s_n-1)+\sum_{j=0}^nb_j)=\mathcal{O}_{X^c}((s_n-3)/4
)$ by adjunction formula. Note that $l\leq 0$ is equivalent to 
$\nu_n\mathrm{deg}(D^c)\leq \mathrm{deg}(-K_{X^c})=(-K_{X^c})\cdot c_1(\mathcal{O}(1))^{n-2}=(s_n-3)/4\cdot (s_n-1)/(b_0\cdots b_n).
$
That is to say, $l\leq 0$ iff $$\nu_n\leq\frac{(s_n-3)(s_n-1)}{4b_0\cdots b_n\mathrm{deg}(D^c)}.$$
 The left side of this inequality is at least $(s_n-3)a_n/(4b_n)$
since $\mathrm{deg}(D^c)=\mathrm{mult}_b{D}\leq (s_n-1)/(b_0\cdots b_{n-1} a_n)$.
Thus we have $l\leq 0$ when 
\begin{equation}\label{nu_n,l<=0,odd}
 \nu_n\leq \frac{(s_n-3)a_n}{4b_n}=\frac{(s_n-3)a_n}{s_n+1},
 \end{equation}
where $(s_n-3)a_n/(s_n+1) \sim s_n^2/4$ since $a_n \sim s_n^2/4$. 

Now we want to find the inequality satisfied by $\nu_n$ so that $(X^c,\nu_nD^c)$ is lc. Note that $b_0>b_1>b_n>b_2>\cdots>b_{n-1}$. By \Cref{multi-bound}, since dimension of $D^c$ is $n-2$, the multiplicity of $D^c$ at every point is at most $b_0b_1\mathrm{deg}(D^c)$ when $n=3$, or $b_{n}b_0b_1\cdots b_{n-3}\mathrm{deg}(D^c)$ when $n\geq 5$, where $b_nb_0\cdots b_{n-3}$ is the product of the $n-1$ largest weights..
Since $\mathrm{deg}(D^c)\leq (s_n-1)/(b_0\cdots b_{n-1} a_n)$, the multiplicity of $D^c$ at every point is at most $(s_3-1)/(b_2a_3)=7/493$ when $n=3$, or $(b_n(s_n-1))/(b_{n-2}b_{n-1}a_n)$ when $n\geq 5$. Therefore, the multiplicity of $\nu_nD^c$ at every point is at most $1$ if 
\begin{equation}\label{nu_n,multiplicityD^c<=1,odd}
\nu_n\leq 
    \begin{dcases}
        \frac{493}{7} \text{ for } n=3,\\
        \frac{b_{n-2}b_{n-1}a_n}{b_n(s_n-1)}=\frac{4a_n(s_n-1)}{s_{n-2}s_{n-1}(s_n+1)} \text{ for } n\geq 5
    \end{dcases}
 \end{equation}
 We have $4a_n(s_n-1)/(s_{n-2}s_{n-1}(s_n+1)) \sim s_n^2/(s_{n-2}s_{n-1}) \sim s^5_{n-2}$ since $s_n-1\sim s_n\sim s^2_{n-1}\sim s_{n-2}^4$ and $a_n\sim s_n^2/4\sim s_{n-2}^8/4$.
Since the stack $X^c$ is a smooth outside the point $[x_0:\cdots:x_{n-1}:x_{n}]=[0:\cdots:0:1]$, then by Theorem \ref{claim2.10.4},  $((X^c,\nu_nD^c))$ is lc outside this point if $\nu$ satisfies inequality \eqref{nu_n,multiplicityD^c<=1}.

Now we want to find the inequality satisfied by $\nu_n$ such that $(X^c,\nu_nD^c)$ is lc near the point $[x_0:\cdots:x_{n}]=[0:\cdots:0:1]$ in $\mathbb{P}(b_0,\ldots,b_n)$. 
In the affine coordinate chart $x_n=1$, the hypersurface $X^c$ is defined by $x_0^2+x_1^3+\cdots+x_{n-1}^{s_{n-1}}+x_1\cdots x_{n-1}=0$. Note that $b_i=c_i$ for $i=0,\ldots,n-1$, for the values of $c_i$ appearing in \cite[Lemma 5.1]{Totaro23}. Hence the pair $(X^c,\nu_nD^c)$ is lc near the point $[x_0:\cdots:x_{n-1}:x_{n}]=[0:\cdots:0:1]$ if
\begin{equation*}
    \nu_n\leq 
    \frac{2}{s_{n-1}^{n-2}(s_{n-1}+1)^2(s_{n-1}-1)^{n-4}\mathrm{mult}_b{D^c}} \text{ for } n\geq 3.
   \end{equation*}
We know $D^c$ corresponds to a cycle with codimension $2$ in this affine coordinate chart $\mathbb{A}^n$ and by \cite[Lemma 4.4]{Totaro23} $$\mathrm{mult}_b{D^c}\leq b_n \mathrm{deg}(D^c)=\frac{b_n(s_n-1)}{b_0\cdots b_{n-1}a_n}=\frac{s_n+1}{4(s_n-1)^{n-2}a_n}.$$
Thus $(X^c,\nu_nD^c)$ is lc near the point $[x_0:\cdots:x_{n}]=[0:\cdots:0:1]$ if
\begin{equation}\label{nu_n,lcatpoint,odd}
\nu_n\leq \frac{8a_n(s_n-1)^{n-2}}{s_{n-1}^{n-2}(s_{n-1}+1)^2(s_{n-1}-1)^{n-4}(s_n+1)}\sim 2s_{n-1}^2 \text{ for } n\geq 3
    \end{equation}
 since $8a_n(s_n-1)^{n-2}\sim 8 s_{n-1}^4(s_n-1)^{n-2}/{4} = 2s_{n-1}^{2n}$ and $s_{n-1}^{n-2}(s_{n-1}+1)^2(s_{n-1}-1)^{n-4}(s_n+1)\sim s_{n-1}^{2n-4}s_{n-1}^2=s_{n-1}^{2n-2}$. 
Among all the right sides of inequalities \eqref{nu_n,multiplicityX<=1,odd}, \eqref{nu_n,l<=0,odd}, \eqref{nu_n,multiplicityD^c<=1,odd} and \eqref{nu_n,lcatpoint,odd}, the one in inequality \eqref{nu_n,multiplicityX<=1,odd} is the smallest one and greater than $1$. 
We know $\alpha$-invariant of $X$ is the supremum of the real numbers $\nu_n$ such that $(X,\nu_nD)$ is lc for every effective $\mathbb{Q}-$divisor $D\sim_{\mathbb{Q}}-K_X$. Hence the $\alpha$-invariant  $\alpha(X)\geq (a_n a_{n+1})/d>1$.

Now we show $\alpha(X)\leq (s_n-2)a_{n+1}/(s_n-1)$.
Indeed, there exists a $\mathbb{Q}$-divisor $D\in |-K_X|_{\mathbb{Q}}$ such that $\mathrm{lct}(X,D)=(s_n-2)a_{n+1}/(s_n-1)$. 
Let $D \coloneqq (1/a_{n+1})E$, where $E$ is the hyperplane section in $X$ given by $\{x_{n+1}=0, x_0^2+\cdots+x_{n-1}^{s_{n-1}}=0\}$. The stack $E$ has one singular point $\p_n \coloneqq  [x_0:\cdots:x_{n-1}:x_n:x_{n+1}]=[0:\cdots:0:1:0]$ which is a smooth point of stack $X$. In the affine chart $x_n=1$, the stack $E$ is \'etale-locally isomorphic to the Fermat-type hypersurface singularity $x_0^2+\cdots+x_{n-1}^{s_{n-1}}=0$ in $\mathbb{A}^n$. By \cite[Example 8.15]{Kollar97}, the log canonical threshold of the pair $(X,E)$ at the point $\p_n$ is $\mathrm{min}\{1/2 +\cdots+1/s_{n-1},1\}=(s_n-2)/(s_n-1)$. Hence the log canonical threshold of $(X, D)$ at $\p_n$ is $(s_n-2)a_{n+1}/(s_n-1)$. Since $\p_n$ is the only singular point in $E$, we get $\mathrm{lct}(X,D) = (s_n-2)a_{n+1})/(s_n-1)$. So $\alpha(X)\leq \mathrm{lct}(X,D) = (s_n-2)a_{n+1}/(s_n-1)$. We have $(s_n-2)a_{n+1}/(s_n-1)\sim s_n^3/16$ since $a_{n+1}\sim s_n a_n/4\sim s_n^3/16$.
\end{proof}

\begin{remark}
When $n=2$, our example is the hypersurface $X_{282} \subset \mathbb{P}(141,94,13,35)$ defined by the equation
$$x_0^2 + x_1^3 + x_2^{19} x_3 + x_1 x_2 x_3^5 = 0.$$
Let $C$ be the curve in $X_{282}$ defined by $\{x_2=0, x_0^2+x_1^3=0\}$. By taking a resolution, we can compute $\mathrm{lct}(X,C)=3/4$. Hence for $D=(1/13)C \in |-K_X|_{\mathbb{Q}}$, we have $\mathrm{lct}(X,D)=39/4$, which gives a sharper upper bound for $\alpha(X)$ than in \Cref{thm: exceptionality}. This upper bound is smaller than the $\alpha$-invariant of $21/2$ achieved by another exceptional del Pezzo surface constructed by Totaro \cite[Section 8]{Totaro23}, which he conjectures to be optimal. That example is another non-quasismooth hypersurface $X_{154}\subset \mathbb{P}(77,45,19,14)$.

Unlike in the small anti-canonical volume examples from \cite{Totaro23}, the divisor computing the $\alpha$-invariant in our examples likely passes through
the non-quasismooth point, and the exact value of the invariant seems to be very difficult to compute in higher dimensions.  Since our examples do not appear to have optimal $\alpha$-invariant (this is true for $n = 2$, for instance, from the last paragraph), we only demonstrate bounds on $\alpha(X)$ in this section.
\end{remark}

\end{document}